\documentclass[2p, a4paper, final, english, preprint]{elsarticle} 
\usepackage[left = 2.5cm, right = 2.50cm, top = 2.5cm, bottom = 2.5cm]{geometry}
\makeatletter
\def\ps@pprintTitle{%
}
\makeatother

\usepackage[utf8]{inputenc}
\usepackage[T1]{fontenc}
\usepackage{amsmath, amsthm, amsfonts, amssymb, mathabx, mathrsfs, mathtools}
\usepackage{dsfont}
\usepackage{graphicx}
\usepackage{comment}
\usepackage{enumitem}
\usepackage{accents}
\usepackage{xcolor}
\usepackage{subfig,adjustbox}
\newcommand{\todo}[1]{} 

\usepackage{hyperref}
\usepackage[nameinlink,capitalize]{cleveref}
\hypersetup{%
	pdftoolbar=true,       							
	pdfmenubar=true,        						
	pdffitwindow=true,     							
	pdfstartview={FitH},    						
	pdftitle={Numerical approximation of the value of a stochastic differential game with asymmetric information},
	pdfsubject={SFB 1283/B3-C4},   						
	pdfkeywords= {zero-sum stochastic differential games} {asymmetric information} {probabilistic numerical approximation} {discrete convex envelope} {convexity constrained Hamilton--Jacobi--Bellmann equation} {viscosity solution}, 
	pdfnewwindow=true,      						
	colorlinks=true, 								
	anchorcolor=black,
}
\crefname{lem}{Lemma}{Lemmas}
\crefname{defn}{Definition}{Definitions}
\crefname{prop}{Proposition}{Propositions}
\newcommand\la{\langle}
\newcommand\ra{\rangle}

\usepackage{tikz}
\usepackage{pgfplots}
\pgfplotsset{compat=1.13}
\usepackage{pgfplotstable}
\pgfplotscreateplotcyclelist{MyColors}{%
	{red},
	{green,mark = square*,every mark/.append style={solid,scale=0.1,fill=red}},
}
\pgfplotsset{compat=1.13}
\usepackage{pgfplotstable}
\pgfplotscreateplotcyclelist{MyColors}{%
	{red,mark = *,every mark/.append style={solid,scale=0.4,fill=red}},
	{teal,mark = x,every mark/.append style={solid,scale=0.4,fill=teal}},
	{blue,mark = square*,every mark/.append style={solid,scale=0.4,fill=blue}},
	{every mark/.append style={solid,scale=0.4,fill=red}},
	{dashed,teal,every mark/.append style={solid,scale=0.4,fill=teal}},}

\pgfplotscreateplotcyclelist{MyColors1}{%
	{red,mark = *,every mark/.append style={solid,scale=0.4,fill=red}},
	{teal,mark = x,every mark/.append style={solid,scale=0.4,fill=teal}},
	{blue,mark = square*,every mark/.append style={solid,scale=0.4,fill=blue}},
	{red,dotted,mark = *,every mark/.append style={solid,scale=0.4,fill=red}},
	{teal,dotted,mark = x,every mark/.append style={solid,scale=0.4,fill=teal}},
	{blue,dotted,mark = square*,every mark/.append style={solid,scale=0.4,fill=blue}},}

\pgfplotscreateplotcyclelist{MyColors2}{%
	{red,mark = *,every mark/.append style={solid,scale=0.4,fill=red}},
	{teal,mark = x,every mark/.append style={solid,scale=0.4,fill=teal}},
	{red,dotted,mark = *,every mark/.append style={solid,scale=0.4,fill=red}},
	{teal,dotted,mark = x,every mark/.append style={solid,scale=0.4,fill=teal}},}
\usepackage[utf8]{inputenc}
\usepackage[T1]{fontenc}
\usepackage{amsmath, amsthm, amsfonts, amssymb, mathabx, mathrsfs, mathtools}
\usepackage{dsfont}
\usepackage{graphicx}
\usepackage{comment}
\usepackage{enumitem}
\usepackage{accents}
\usepackage{xcolor}
\usepackage{subfig,adjustbox}
\newcommand{\hp}{{h}}
\newcommand{\hx}{{\Delta x}}
\newcommand{\hht}{{\tau}}
\newcommand\ot{{\widebar{t}}}
\newcommand\oX{{\widebar{X}}}
\newcommand\oY{{\widebar{Y}}}
\newcommand\oZ{{\widebar{Z}}}
\newcommand\op{{\widebar{p}}}
\newcommand{\be}{\mathrm{\psi}}

\newcommand\mH{{\mathcal{H}}}

\newcommand\mN{{\mathcal{N}}}
\newcommand\mM{{\mathcal{M}}}
\newcommand\mO{{\mathcal{O}}}

\newcommand\mU{{\mathcal{U}}}
\newcommand\mV{{\mathcal{V}}}

\newcommand{\mT}{\mathcal{T}}
\newcommand{\mY}{\mathcal{Y}}
\newcommand{\mX}{\mathcal{X}^{\Delta x}}
\newcommand\EE{\mathbb{E}}
\newcommand\RR{\mathbb{R}}
\newcommand\PP{\mathbb{P}}

\newcommand\UU{\mathbb{U}}
\newcommand\VV{\mathbb{V}}

\newcommand\sig{{\sigma}}

\newcommand\lbd{{\lambda}}
\newcommand\lbdmin{{\lbd_{\min}}}
\newcommand{\Tr}{\mathrm{Tr}}
\newcommand{\vexp}{\mathrm{Vex}_p}

\newcommand{\Int}{{\mathrm{Int}}}
\newcommand{\oxn}{{\widebar{x}_n}}

\newcommand{\ott}{{\widebar{t}_\hht}}
\newcommand{\oxt}{{\widebar{x}_\hht}}

\newcommand{\opt}{{\widebar{p}_\hht}}
\newcommand{\ox}{\widebar{x}}
\newcommand{\xl}{{x_\ell}}
\newcommand{\tp}{\tilde{p}}
\DeclareMathOperator{\diam}{\mathrm{diam}}
\newcommand{\I}{\texttt{I}}
\newcommand{\II}{\texttt{II}}
\newcommand{\III}{\texttt{III}}
\newcommand{\IV}{\texttt{IV}}
\newcommand{\fp}{\boldsymbol{\mathsf{p}}}

\newcommand{\bi}{\boldsymbol{\mathsf{i}}}

\newcommand\tn{{t_n}}

\newcommand{\reff}{{\mathrm{ref}}}

\newcommand{\signTlx}{\sigma_{n}^{-T}(x)}
\newcommand{\signTlxp}{\sigma_{n}^{-T}(x')}
\newcommand{\signTl}{\sigma_{n}^{-T}}
\newcommand{\ind}{\mathds{1}}
\newcommand{\llvert}{\left\lvert}
\newcommand{\rrvert}{\right\rvert}
\newcommand{\dt}{\partial_t}
\newcommand{\D}{\displaystyle}
\newcounter{gr111}
\newenvironment{steps}
{\begin{list} {\bf Step \arabic{gr111}.$\, $} {\usecounter{gr111}
			\setlength{\labelwidth}{-0.2cm} \setlength{\leftmargin}{0.0cm}
			\setlength{\topsep}{0.1cm} \setlength{\itemsep}{0.2cm}
			\setlength{\parsep}{0.1cm} \setlength{\itemindent}{0.4cm}
			\setlength{\parskip}{0.0cm}}} 
	{\end{list}}
\newcounter{insert}
\newenvironment{assertion}
{\begin{list} {(\roman{insert})$\,$} {\usecounter{insert}
			\setlength{\labelwidth}{-0.2cm} \setlength{\leftmargin}{0.0cm}
			\setlength{\topsep}{0.1cm} \setlength{\itemsep}{0.2cm}
			\setlength{\parsep}{0.1cm} \setlength{\itemindent}{0.4cm}
			\setlength{\parskip}{0.0cm}}} 
	{\end{list}}
\crefname{assertion}{item}{item}
\newtheorem{theorem}{Theorem}[section]
\newtheorem{defn}[theorem]{Definition}
\newtheorem{lem}[theorem]{Lemma}

\newtheorem{prop}[theorem]{Proposition}

\newtheorem{algo}{Algorithm}[section]
\newtheorem{rem}[theorem]{Remark}

\numberwithin{equation}{section}

\begin{document}
	\begin{frontmatter}
		\title{Numerical approximation of the value of a stochastic differential game with asymmetric information\tnoteref{t1}}
		
		\tnotetext[t1]{This research was supported by the German Research Foundation as part of the Collaborative Research Center SFB1283. \flushright \today}
		
		\author{\v{L}ubom\'ir Ba\v{n}as\fnref{a1}}
		\author{Giorgio Ferrari\fnref{a2}}
		\author{Tsiry Avisoa Randrianasolo\fnref{a1}}
		
		\address[a1]{Faculty of Mathematics, Bielefeld University, Universit{\"a}tsstr.\,25, D-33615 Bielefeld}
		\address[a2]{Center for Mathematical Economics, Bielefeld University, Universit{\"a}tsstr. 25, D-33615 Bielefeld}
		
		\begin{abstract}
			We consider a convexity constrained Hamilton--Jacobi--Bellman-type obstacle problem for the value function of a zero-sum differential game with asymmetric information.
			We propose a convexity-preserving probabilistic numerical scheme for the approximation of the value function which is discrete w.r.t.\ the time and convexity variables,
			and show that the scheme converges to the unique viscosity solution of the considered problem.
			Furthermore, we generalize the semi-discrete scheme to obtain an implementable fully discrete numerical approximation of the value function
			and present numerical experiments to demonstrate the properties of the proposed numerical scheme.
		\end{abstract}
		
		
	\end{frontmatter}
	
	\section{Introduction}\label{sec:0}
	{  
		In this paper we consider the Hamilton--Jacobi--Bellman-type obstacle problem
		\begin{equation}\label{eq:1}
		\begin{split}
		\min&\left\{ \dt V + \frac12\Tr(\sigma\sigma^T(t,\;x)D^2_x V) + H(t,\;x,\;D_x V,\;p) ,\,\lambda_{\min}\left(p,\;\frac{\partial^2 V}{\partial p^2}\right)  \right\} = 0,
		\\
		V&(T,\;x,\;p) = \la p,\;g\ra,
		\end{split}
		\end{equation}
		where $V\equiv V(t,x,p)$,  $(t,x,p)\in [0,T]\times \RR^d\times \Delta(I)$, $\Delta(I)$ denotes the set of probability vectors $p = (p_1,\ldots,p_I)\in{(0,1)^I}$ that satisfy $\sum_{i=1}^{I}p_i = 1$
		and the Hamiltonian $H$ will be specified below.
		The convexity of the solution $V$ with respect to the variable $p$ is enforced via the obstacle term $\lambda_{\min}\left(p,\;\frac{\partial^2 V}{\partial p^2}\right)$, which
		is the minimal eigenvalue of the Hessian matrix  $\frac{\partial^2 V}{\partial p^2}$ on the tangent cone to $\Delta(I)$. More precisely, for a symmetric $I\times I$ matrix $A$ we denote
		\begin{equation*}
		\lbdmin (p,\;A) :=\min_{z\in T_{\Delta(I)(p)\setminus\{0\}}}\frac{\la A z, \;z\ra}{\lvert z\rvert^2},
		\end{equation*}
		where $T_{\Delta(I)(p)} = \overline{\bigcup_{\delta>0}(\Delta(I)-p)/\delta}$ is the tangent cone to $\Delta(I)$ at $p\in\Delta(I)$, cf.\ \cite{cardaliaguet2009a}.
	}
	
	{Problem (\ref{eq:1}) describes the value of a class of zero-sum stochastic differential games with asymmetric information, cf.\ \cite{gruen2012aprobabilistic}.} Since the seminal work by Aumann and Maschler (see \cite{maschler1995repeated}) in the framework of repeated games, the literature on games with asymmetric information experienced an increasing interest (\cite{meyer1999cav, mertens1971the, rosenberg2004stochastic}, among many others), recently also in continuous-time differential settings (see, e.g., \cite{cardaliaguet2009stochastic, cardaliaguet2013pathwise, cardaliaguet2012a, gensbittel2018a, gruen2013on, wu2017existence, wu2018existence}).
	
	As in \cite{gruen2012aprobabilistic}, in our game both players can adjust the dynamics of a non degenerate It\^o-diffusion by controlling the drift via regular controls taking values in some compact subset of a finite dimensional space. However, one player has more information than the other in the following sense (cf.\ \cite{maschler1995repeated} and \cite{cardaliaguet2009stochastic}). Before the game starts, the payoffs of the game are chosen randomly with some probability $p$ from a finite collection of size $I$, and the information on which payoffs have been realized is transmitted only to one player. Since we assume that both players can observe the actions of the other one, the uninformed player infers which game is actually played through the moves of the informed one. It turns out that it is optimal for the informed player to release information to the uninformed one in a sophisticated way aiming at manipulating the beliefs of the latter player (see \cite{cardaliaguet2009stochastic}). 
	
	The numerical analysis of our paper hinges on the theoretical results of \cite{cardaliaguet2009stochastic}.
	There it is shown (in a setting actually more general than ours) that the previously described game has a value $V$, whenever the so-called Isaacs conditions are satisfied and additional technical requirements on the problem’s data area fulfilled. 
	{The value function $V$ depends on time $t$, on the state variable $x$, and on a probability vector $p\in \Delta(I)$;} this latter variable describes the initial value of the beliefs of the uninformed player about the game she is playing. Moreover, it is shown in 
	\cite{cardaliaguet2009a}, that $V$ can be characterized as the unique continuous viscosity solution (in the dual sense) to a second-order partial differential equation complemented by a convexity constraint with respect the parameter $p$.

	There exist only few results on numerical approximation of differential games with incomplete information.
	Numerical approximation of (deterministic) differential games with incomplete information was first studied in \cite{cardal09num}
	and generalized to games with incomplete information on both sides in \cite{souquiere2010approximation}. As far as we are aware the only work on numerical approximation
	of stochastic differential games with incomplete information is \cite{gruen2012aprobabilistic}.
	We note that all three aforementioned works only consider semi-discretization in the time-variable and the remaining variables are kept continuous, hence, the schemes are not implementable.

	In this paper we generalize the probabilistic numerical approximation of \cite{gruen2012aprobabilistic} to include the discretization of the convex envelope, i.e.,
	we propose a structure preserving probabilistic numerical approximation that 
	is discrete in time and in the variable $p$ and preserves the convexity of the solution.
	We show that the proposed numerical approximation converges to the unique viscosity solution of (\ref{eq:1}).
	{The discretization in the probability variable $p$ is constructed by approximating the lower convex envelope of the semi-discrete solution in $p$
		by its finite-dimensional counterpart. The discrete lower convex envelope is computed over a finite set of values which coincide with nodes of a simplicial partition of $\Delta(I)$.
		The resulting approximation is monotone and inherits the Lipschitz continuity properties of the solution.
	}
	{To further reduce the complexity of the numerical approximation we employ random walk
		instead of the usual Wiener increments to simulate the associated It\^o-diffusion process.
	}
	Furthermore, we propose an implementable fully discrete numerical scheme by combining the semi-discrete probabilistic approximation in time and $p$
	with a spatial discretization that employs linear interpolation in the state variable $x$ over a simplicial partition.
	
	The paper is organized as follows.
	In Section~\ref{sec2} we collect basic definitions and assumptions on the considered problem.
	In Section~\ref{sec_num} we introduce a probabilistic numerical scheme for the approximation of (\ref{eq:1})
	which is discrete in the time variable $t$ and the convexity variable $p$ and summarize the regularity properties
	of the numerical approximation in Section~\ref{sec_reg}.
	Convergence of the numerical approximation to the viscositiy solution is shown in Section~\ref{sec_conv}.
	Finally, an implementable fully discrete numerical approximation of the problem 
	is introduced in Section~\ref{sec_comput} along with
	numerical studies which demonstrate the practicability of the proposed approach.

	\section{Assumptions and preliminaries}\label{sec2}

	Throughout the paper, the scalar product of two vectors $x = (x_1,\ldots,x_d)$ and $y= (y_1,\ldots,y_d)$ of $\RR^d$ is denoted by $\la x, \; y\ra  := \sum_{i=1}^{d}x_iy_i$ 
	and the $\ell^1$-norm is denoted by $\lvert x\rvert:=\sum_{i=1}^{d}\lvert x_i\rvert$; furthermore, we use $|\cdot|_{\infty}$ and $\|\cdot\|_{\infty}$ to respectively denote the $\ell^\infty$-norm and the $L^\infty(\mathbb{R}^d)$-norm.
	
	\subsection{Description of the game}
	
	{{Since the aim of this paper is to provide a numerical approximation of the solution to (\ref{eq:1}), we only provide here a brief and informal description of the stochastic differential game related to the problem (\ref{eq:1}) and simply refer to \cite{gruen2012aprobabilistic} for detailed discussion of the game and further references.}}
	We consider a two-player zero-sum differential game where two players control the $d$-dimensional It\^o process defined for $t\in[0,\;T]$, $x\in\RR^d$ as
	\begin{equation}\label{eq:state}
	\begin{array}{rcl}
	\D dX^{t,\; x,\;  u,\;  v}_s & = & \D b(s, \; X^{t,\;  x, \; u,\;  v}_s,\; u_s,\;  v_s)ds + \sig(s,\; X^{t,\;  x, \; u,\;  v}_s) d B_s\qquad s\in [t,\;T]\,,\\ 
	\D X^{t, \; x, \; u,\;  v}_t & = & x.
	\end{array}
	\end{equation}
	Here $B = \big\{B_s:s\in[t,\;T]\big\}$ is a $d$-dimensional Brownian motion {{on a complete probability space}}, $b$ and $\sigma$ are suitable Borel-Measurable functions and the controls $(u,\;v)\in \UU\times\VV$ and $\UU,\;\VV$ are compact subsets of some finite dimensional spaces.
	
	The game is characterized by $I$ configurations with respective running costs $(\ell_i)_{i\in \left\{1,\ldots,I\right\}}:[0,\;T]\times\RR^d\times \UU \times \VV\rightarrow\RR$
	and terminal payoffs $(g_i)_{i\in \left\{1,\ldots,I\right\}}:\RR^d\rightarrow\RR$
	and is played as follows.
	Before the game starts, one configuration $i\in \{ 1, \dots,  I\}$ is chosen with probability $p_i$
	and the choice of $i$ is communicated to Player 1. Player 2 only knows the probability distribution $p\in \Delta(I)$ of the respective configurations.
	Once the game has started, both players adjust their control to minimize, for the Player 1, and to maximize, for the Player 2, the {expected payoff}, cf. \cite[Section~6.3]{cardaliaguet2009a}.  
	We assume that both players observe their opponent's control. 
	
	\subsection{General assumptions}\label{ssec:11}
	The drift term $b$, the diffusion term $\sigma :=\left(\sigma_{k,l}\right)_{{k,l}} $, the running cost $\left(\ell_i\right)_{i\in\left\{1,\ldots,I \right\}}$, the terminal payoff $g:=\left(g_i\right)_{i\in\left\{1,\ldots,I \right\}}$, and the Hamiltonian $H$, cf. \eqref{eq:1}, satisfy the following assumptions:
	
	\crefname{enumi}{Assumption}{Assumptions}
	\begin{enumerate}[label=($A_\arabic*)$]
		\item $b:[0,\;T]\times\RR^d\times \UU \times \VV\rightarrow\RR^d$ is bounded and continuous in all its variables and Lipschitz continuous with respect to $(t,\;x)$ uniformly in $(u,\;v)\in \UU \times \VV$.
		\label{A1}
		\item For $1\leq k,l\leq d$ the function $\sigma_{k,l}:[0,\;T]\times\RR^d\rightarrow\RR$ is bounded and Lipschitz continuous with respect to $(t,\;x)$. For any $(t,\;x)\in [0,\;T]\times \RR^d$ the matrix $(\sigma^T)^{-1}$, where the superscript $T$ means transpose, is non-singular, bounded, and Lipschitz continuous with respect to $(t,\;\;x)$.
		\label{A2}
		\item $(\ell_i)_{i\in \left\{1,\ldots,I\right\}}:[0,\;T]\times\RR^d\times \UU \times \VV\rightarrow\RR$ is bounded and continuous in all its variables and Lipschitz continuous with respect to $(t,\;x)$ uniformly in $(u,\;v)$. 
		
		\noindent$(g_i)_{i\in \left\{1,\ldots,I\right\}}:\RR^d\rightarrow\RR$ is bounded and uniformly Lipschitz continuous.
		\label{A3}
		\item \textit{Isaacs condition}: for all $(t,\;x,\;z,\;p)\in [t_0,\;T]\times\RR^d\times\RR^d\times\Delta(I)$
		\begin{equation*}
		\begin{split}
		H(t,\;x,\;z,\;p)&:=\inf_{u\in\UU}\sup_{v\in\VV} \bigg\{\la b(t,\;x,\;u,\;v),\;z\ra + \sum_{i=1}^{I}p_i\ell_i(t,\;x,\;u,\;v)\bigg\} 
		\\
		&= \sup_{v\in \VV}\inf_{u\in \UU} \bigg\{\la b(t,\;x,\;u,\;v),\;z\ra + \sum_{i=1}^{I}p_i\ell_i(t,\;x,\;u,\;v)\bigg\}.
		\end{split}
		\end{equation*}
		\label{A4}

	\end{enumerate}
	We note that as a consequence of Assumptions~\ref{A1}-\ref{A3} there exists a constant $C>0$ such that for all $t,t'\in[0,\;T],\,x,x'\in\RR^d,\;\,z,z'\in\RR^d,\;\,p,p'\in\Delta(I)$, the following hold
	\begin{equation}\label{eq:3}
	\lvert H(t,\;x,\;z,\;p)\rvert\leq C(1+\lvert z\rvert),
	\end{equation}
	\begin{equation}
	\label{eq:4}
	\begin{split}
	\lvert H(t,\;x,\;z,\;p)-H(t',\;x',\;z',\;p')\rvert
	&\leq C(1+\lvert z\rvert)(\lvert x-x'\rvert + \lvert t-t'\rvert)
	\\
	&\hspace{10pt}+ C\lvert z-z'\rvert + C\lvert p-p'\rvert.
	\end{split}
	\end{equation}
	
	\subsection{Viscosity solution of \eqref{eq:1}}
	Under the assumptions in the previous section Cardaliaguet \cite{cardaliaguet2009a,cardaliaguet2009stochastic} 
	established that {{there exists a unique uniformly bounded viscosity solution of problem (\ref{eq:1}), which is convex and uniformly Lipschitz continuous in p.}}
	We recall the notion of viscosity solution as well as the corresponding notions of subsolutions and supersolutions to \eqref{eq:1} below, cf. \cite{cardaliaguet2009a}, \cite{cardaliaguet2009on}.
	
	\begin{defn}
		We say that $V$ is a subsolution of \eqref{eq:1} if $V = V(t,\;x,\;p)$ is upper semicontinuous and if, for any smooth test function $\phi:(0,\;T)\times\RR^d\times\Delta(I)\rightarrow\RR $ such that $V-\phi$ has a local maximum on $[0,\;T]\times\RR^d\times\Delta(I)$ at some point $(\ot,\;\ox,\;\op)\in[0,\;T]\times\RR^d\times\Delta(I)$, one has
		\begin{equation}\label{eq:subv}
		\min\left\{ \dt \phi + \frac12\Tr(\sigma\sigma^T(t,\;x)D^2_x \phi) + H(t,\;x,\;D_x \phi,\;p) ,\,\lambda_{\min}\left( p,\;\frac{\partial^2\phi}{\partial p^2}\right)  \right\} \geq 0,
		\end{equation}
		at $(t,\;x,\;p)=(\ot,\;\ox,\;\op)$.
		
		We say that $V$ is a supersolution of \eqref{eq:1} if $V = V(t,\;x,\;p)$ is lower semicontinuous and if, for any smooth test function $\phi:(0,\;T)\times\RR^d\times\overline{\Delta(I)}\rightarrow\RR $ such that $V-\phi$ has a local minimum on $[0,\;T]\times\RR^d\times\overline{\Delta(I)}$ at some point $(\ot,\;\ox,\;\op)\in[0,\;T]\times\RR^d\times \Delta(I)$, one has
		\begin{equation}\label{eq:supv}
		\min\left\{ \dt \phi + \frac12\Tr(\sigma\sigma^T(t,\;x)D^2_x \phi) + H(t,\;x,\;D_x\phi,\;p) ,\,\lambda_{\min}\left(p,\;\frac{\partial^2\phi}{\partial p^2}\right)  \right\}  \leq 0,
		\end{equation}
		at $(t,\;x,\;p)=(\ot,\;\ox,\;\op)$.
		
		We say that $V$ is a viscosity solution of \eqref{eq:1} if $V$ is a sub- and  a supersolution of \eqref{eq:1}.
	\end{defn}
	
	\begin{rem}
		For a smooth test function $\phi:(0,\;T)\times\RR^d\times\Delta(I)\rightarrow\RR $ such that $V-\phi$ has a local maximum on $[0,\;T]\times\RR^d\times\Delta(I)$ at some point $(\ot,\;\ox,\;\op)\in[0,\;T]\times\RR^d\times\Delta(I)$, we have that \eqref{eq:subv} is equivalent to
		\begin{align}\label{eq:13}
		\dt \phi + \frac12\Tr(\sigma\sigma^T(t,\;x)D^2_x \phi) + H(t,\;x,\;D_x \phi,\;p)
		&\geq 0\quad\mbox{and}\quad\lambda_{\min}\left(p,\frac{\partial^2\phi}{\partial p^2}\right)\geq 0;
		\intertext{ and for the smooth test function $\phi:(0,\;T)\times\RR^d\times\Delta(I)\rightarrow\RR $ such that $V-\phi$ has a local minimum on $[0,\;T]\times\RR^d\times\overline{\Delta(I)}$ at some point $(\ot,\;\ox,\;\op)\in[0,\;T]\times\RR^d\times \overline{\Delta(I)}$, \eqref{eq:supv} is equivalent to}
		\label{eq:14}
		\dt \phi + \frac12\Tr(\sigma\sigma^T(t,\;x)D^2_x \phi) + H(t,\;x,\;D_x \phi,\;p)
		&\leq 0\quad\mbox{ or}\quad\lambda_{\min}\left(p,\;\frac{\partial^2\phi}{\partial p^2}\right)\leq 0.
		\end{align}
	\end{rem}

	\section{Numerical approximation}\label{sec_num}
	
	
	{{To simplify the subsequent numerical approximation, we perform a change of measure via the Girsanov transform in the spirit of, for instance, \cite[Lemma~3.8]{gruen2012a} or \cite{hamadene1995zero}, and instead of the controlled process (\ref{eq:state}) we consider the simpler process 
			\begin{equation}\label{eq:state1}
			\begin{array}{rcl}
			\D dX^{t,\;x}_s  & = & \sig(s,\;X^{t,\;x}_s) d B_s\qquad s\in [t,T]\,,\\
			\D X^{t,\;x}_t & = & x,
			\end{array}
			\end{equation}
			for $t\in[0,\;T]$ and $x\in\RR^d$. Notice that the dynamics in \eqref{eq:state1} is independent of the players' controls.}}
	
	For a fixed $N\in \mathbb{N}$ and a step size $\hht := T/N$ we introduce an equidistant partition $\Pi_{\hht} := \big\{t_n\big\}_{n=0}^N$, $t_n = n\tau$ of the time interval $[0,\;T]$.
	We define the discrete process $( \oX_{n}^{n',\;x})_{n'=n,\ldots,\;N}$ as the weak Euler approximation of \eqref{eq:state1}, that is
	\begin{equation}\label{eq:Xeuler}
	\oX_{n}^{n',\;x} = x +\sum_{j=n'}^{n-1}\sigma(t_j,\;\oX_{j}^{n',\;x})\xi_j\sqrt{\hht},
	\end{equation}
	where $\xi_{n}\sqrt{\hht} = (\xi_{n}^1,\ldots,\;\xi_{n}^d)\sqrt{\hht}$, $n=1,\dots, N$ {is a suitable approximation of the $\RR^d$-valued Wiener increment $[W(t_n)-W(t_{n-1})] \sim \mathcal{N}(0,\tau)$}.
	Here, we take $\xi_{n}$ to be a $\RR^d$-valued binomial random walk, i.e. $\xi_{n}^1,\ldots,\xi_{n}^d$ are i.i.d. random variables with the law $\PP(\xi_{n}^k = \pm 1) = 1/2$, for every $k = 1,\ldots,d$;
	the analysis below can be easily modified to cover other choices such as, e.g., a trinomial random walk or the discrete Wiener increments.
	In the following we abbreviate $\sig_{n}(x):=\sig(t_n,\;x)$ and $\sig^{-T}_n(x):=(\sig^T(t_n,\;x))^{-1}$.
	The approximation obtained after one step of the Euler scheme (\ref{eq:Xeuler}) will be denoted as
	\begin{equation}\label{xeuler1}
	\oX^{x}_{n+1}:=\oX_{n+1}^{n,\;x} = x + \sigma_n(x)\xi_{n}\sqrt{\hht} \qquad x\in\mathbb{R}^d.
	\end{equation}

	Let $\{\mM^{\hp}\}_{h>0}$ be a family of regular partitions of $\Delta(I)$ into open $(I-1)$-simplices $K$
	{(i.e., line segments, triangles, tetrahedra for $I = 2,3,4$, respectively)} with mesh-size $\hp = \max_{K\in\mM^\hp}\{\diam(K)\}$
	such that $\overline{\Delta(I)} = \cup_{K\in\mM^\hp}\overline{K}$. 
	The set of vertices of all $K\in \mM^\hp$ is denoted by $\mN^{h}:=\{p_1,\ldots,p_M\}$. 
	
	%
	The approximation of the value function $V(t_n,\;x,\;p_m)$ is denoted by $V^{m}_n(x)$ for $t_n\in\Pi_\hht,\;x\in\RR^d,\;p_m\in\mN^\hp$.
	The discrete numerical solution $V^{m}_n(x)$, $x\in \RR^d$, $n=0,\dots, N-1$, $m=1,\dots,M$ is obtained by the following algorithm.
	\begin{algo}\label{algo:algo}
		For $x\in \RR^d$ set $V^{m}_N(x) = \la p_m,\;g(x)\ra$ for $p_m \in \mN^{h}$, $m=1,\ldots,M$,  
		set $\mV_{N}(x)=\big\{V^{1}_{N}(x), \ldots, V^{M}_{N}(x)\big\}$ and proceed for $n  = N-1,\ldots,0$ as follows: 
		\begin{enumerate}
			\item Forward step: for $x\in\RR^d$ compute:
			\begin{equation}\label{eq:algo2}
			\oX_{n + 1}^x = x + \sigma_{n}(x)\xi_{n}\sqrt{\hht};
			\end{equation}
			\item Backward step: for $x\in\RR^d$ and $m=1,\ldots,M$ set
			\begin{align}
			\label{eq:algo3}
			\oZ^{m}_{n}(x) 
			&=\frac{1}{{\hht}}\EE\big[V_{n + 1}^{m}(\oX^{x}_{n + 1})\signTlx\xi_{n}\sqrt{\hht}\big],
			\\
			\label{eq:algo4}
			\oY^{m}_{n}(x) &=\EE\big[V^{m}_{n + 1}(\oX^{x}_{n+1})\big] + \hht H\big(\tn,\;x,\;\oZ^{m}_{n}(x),\;p_m\big)\,;
			\end{align}
			%
			\item Convexification: for $x\in\RR^d$ compute the discrete lower convex envelope $\big\{V^{1}_{n}(x), \ldots, V^{M}_{n}(x)\big\}$ of $\big\{\oY^{1}_{n}(x),\ldots,\oY^{M}_{n}(x)\big\}$ as
			\begin{equation}\label{eq:algo5}
			V^{m}_{n}(x) =\vexp\big[\oY^{1}_{n}(x),\ldots,\oY^{M}_{n}(x)\big](p_m)\qquad p_m\in\mathcal{N}^h,\,\, m=1,\ldots,M.
			\end{equation}
		\end{enumerate}
	\end{algo}
	
	We note that for $p\in\overline{\Delta(I)}$ the lower convex envelope in \eqref{eq:algo5} is the solution of the minimization problem, cf. \cite{carnicer1992convexity},
	\begin{equation}\label{eq:vex0}
	\begin{split}
	&\vexp\big[\oY^{1}_{n}(x),\ldots,\oY^{M}_{n}(x)\big](p) := \min\bigg\{\sum_{k=1}^{I}\oY^k_{n}(x)\lambda_k;\,\,\, \sum_{k=1}^{I}\lambda_k=1,\,\,\, \lambda_k\geq 0,\,\,\,\sum_{k=1}^{I}\lambda_k p_{k} = p\bigg\}.
	\end{split}
	\end{equation}
	{We will discuss efficient algorithms for the computation of the discrete lower convex envelope (\ref{eq:vex0}) in Section~\ref{sec_alg_impl}.}

	{It is well known that the piecewise linear interpolation does not preserve the convexity of the interpolated data. 
		Nevertheless, cf.\ \cite[Corollary~2.3.]{carnicer1992convexity}, there exists a data dependent (regular) simplicial partition $\mathcal{M}^h_{n,x}$ of $\Delta(I)$ 
		with nodes $\mN^{h}_{n,x} := \{\pi_{n,x}^1,\dots, \pi_{n,x}^{M_{n,x}}\}\subseteq \mN^{h}$ 
		such that the piecewise linear interpolant of the data values at the nodes $\mN^{h}_{n,x}$ over the partition $\mathcal{M}^h_{n,x}$ (for a precise definition see (\ref{eq:pinter}) below)
		agrees with the discrete data values $\big\{\big(p_m, V_n^m(x)\big)\big \}_{m=1}^M$, $p_m\in\mathcal{N}^h$ of the discrete lower convex envelope (\ref{eq:algo5}) for fixed $0\leq n \leq N$, $x\in \mathbb{R}^d$.
		We note that the partition $\mathcal{M}^h_{n,x}$ does not necessarily coincide with the original mesh $\mathcal{M}^h$.
		
	}
	We consider the set of piecewise linear Lagrange basis functions $\{\be_{n,x}^{i},\,\, i=1,\dots, M_{n,x}\}$ associated with the set of nodes $\mN^{h}_{n,x}$ of the partition $\mathcal{M}^h_{n,x}$.
	We recall the following properties of the the Lagrange basis functions which will be frequently used throughout the paper:  
	$a)$  $\be_{n,x}^{i}(\pi_{n,x}^k) = \delta_{ik}$, where $\delta_{ik}$ is the Kronecker delta and
	$b)$  $\sum_{i=1}^{M_{n,x}}\be_{n,x}^{i}(p) = 1$ for any $p\in\overline{\Delta(I)}$.
	We note that $a)$ implies that at any point $p\in \overline{\Delta(I)}$ there are at most $I$ basis functions with non-zero value at this point, hence the sum in $b)$ reduces to $\sum_{i=1}^{I}$.
	
	We define the convex piecewise linear interpolant $V^{\hp}_n(x,\cdot)$, $x\in\mathbb{R}^d$ of the discrete lower convex envelope $\big\{V^{1}_{n}(x), \ldots, V^{M}_{n}(x)\big\}$ over the convexity preserving partition $\mathcal{M}^h_{n,x}$ as
	\begin{equation}\label{eq:pinter}
	V^{\hp}_n(x,\;p) := \sum_{i=1}^{M_{n,x}}V^{m(\pi_{n,x}^i)}_n(x)\be_{n,x}^{i}(p)\, ,
	\end{equation}
	where $m(\pi_{n,x}^i)\in \mathbb{N}$ is the index of $\pi_{n,x}^i$ in $\mathcal{N}^h$, i.e. $\pi_{n,x}^i = p_{m(\pi_{n,x}^i)}$ for some $p_{m(\pi_{n,x}^i)}\in \mathcal{N}^h$.
	We note that by construction $V^{\hp}_n(x,\;p_m) = V_n^m(x)$ for all $p_m\in\mathcal{N}^h$. For the analysis below we also consider the (possibly non-convex) interpolant over the fixed partition $\mathcal{M}^h$
	\begin{equation}\label{eq:pinter2}
	\widetilde{V}^{\hp}_n(x,\;p) := \sum_{m=1}^{M}V^m_n(x)\be^{m}(p),
	\end{equation}
	where $\{\be^{m},\,\, m=1,\dots, M\}$ is the linear Lagrange basis associated with the set of nodes $\mN^{h}$.
	By a slight abuse of notation in (\ref{eq:vex0}), we observe that
	\begin{equation}\label{vtilde}
	{V}^{\hp}_n(x,\;p) = \vexp\big[\widetilde{V}^{\hp}_n(x,\;\cdot)\big](p).
	\end{equation}

	Furthermore, we  define the time interpolant $V^\hp_\hht(t,\;x,\;p)$ of (\ref{eq:pinter}) which is continuous on $[0,T]$ as
	\begin{equation}\label{eq:tinter}
	V^{\hp}_\hht(t,\;x,\;p) := V^{\hp}_n(x,\;p)\Big(\frac{t_{n+1}-t}{\hht}\Big)+ V^{\hp}_{n+1}(x,\;p)\Big(\frac{t-\tn}{\hht}\Big), \qquad\mathrm{for}\,\,\,\, t \in [\tn,\; t_{n + 1}].
	\end{equation}
	
	\section{Regularity properties of the numerical approximation}\label{sec_reg}
	{In this section we study regularity properties of the numerical approximation obtained by Algorithm~\ref{algo:algo}.
		We establish uniform boundedness, almost H\"older continuity in time and Lipschitz continuity in $p$, and $x$, respectively,
		of the numerical solution. Furthermore, we show that the numerical approximation satisfies a monotonicity property.}
	
	We recall the following properties of the discrete lower convex envelope which are a simple consequence of its definition (\ref{eq:vex0}).
	\begin{lem}\label{prop:monotonicvexm}
		We consider the set $\mN:=\big\{p_1,\ldots,p_M\big\}\subset\overline{\Delta(I)}$ with associated scalar values $U(p_m)$, $V(p_m)$, such that $U(p_m)\leq V(p_m)$,  $m=1,\dots,M$. We denote $\mV=\big\{V(p_1),\ldots,V(p_M)\big\}\in \RR^M$ and $\mU=\big\{U(p_1),\ldots,U(p_M)\big\}\in \RR^M$. 
		The discrete lower convex envelope $\vexp$ satisfies the following properties for $p\in \overline{\Delta(I)}$:
		\smallskip
		\begin{itemize}
			\item[i)] Monotonicity: $\vexp[ \mU](p)\leq \vexp[ \mV](p)$,
			\item[ii)] Constant preservation: $\vexp [\mV + \theta](p) = \vexp[\mV](p) + \theta$ for any $\theta \in\RR$.
		\end{itemize}
	\end{lem}

	\subsection{Lipschitz continuity in $p$}\label{subsec:plip}
	
	\begin{lem}\label{lem:plip0}
		There exists a constant $C>0$ (which only depends on \cref{A1,A2,A3,A4} such that for $n=0,\dots,N$ and all $x\in\RR^d$
		the numerical solution is Lipschitz continuous in the variable $p$, i.e., 
		\begin{equation}\label{eq:plip017}
		\lvert V^{\hp}_n(x,\;p) - V^{\hp}_n(x,\;q)\rvert\leq C\lvert p-q\rvert \qquad \forall p,\;q\in\Delta(I).
		\end{equation}
	\end{lem}
	\begin{proof}
		For $n = N$, 
		by linearity the function $V^\hp_N(x,\;p):=\la p,\;g(x)\ra$ is Lipschitz continuous in $p$ for any $x\in\RR^d$
		with a Lipschitz constant $l_\tau$ that only depends on $g$.
		
		We proceed by induction and assume that $V^{\hp}_{n + 1}(x,\;p)$ is Lipschitz continuous in $p$ with a Lipschitz constant $L_{n + 1}$ for some $n\leq N-2$. 
		We consider $p,\;q$ $\in\Delta(I)$ and assume without loss of generality  that $V^\hp_n(x,\;q)-V^\hp_n(x,\;p)\geq 0$, otherwise $p$ and $q$ can be commuted.
		
		Let $p \in \overline{K}_p$, where $K_p = [\pi_{n,x}^1(p),\ldots,\pi_{n,x}^I(p)] \subset \mathcal{M}_{n,x}^h $ is the simplex given by the nodes
		$\pi_{n,x}^1(p),\ldots,\pi_{n,x}^I(p)\in\mathcal{N}^\hp_{n,x}$. Hence,
		$p=\sum_{i=1}^{I}\pi^i_{n,x}(p)\be^i_{n,x}(p)$, where $\psi_{n,x}^i(p)$, $i=1,\ldots,I$ are the linear Lagrange basis functions on $[\pi_{n,x}^1(p),\ldots,\pi_{n,x}^I(p)]$. 
		
		We note that $\sum_{i=1}^{I}\be^i_{n,x}(p) = 1$ and $\be^i_{n,x}(p)\geq 0 $ for $i=1,\ldots,I$. 
		Hence, by \cite[Lemma\,8.2.]{laraki2004on}, there exist vectors $\big\{\omega^1_{n,x},\ldots,\omega^I_{n,x}\big\}\in\Delta(I)$ 
		(the vectors depend on $p$, $q$, $\mathcal{M}^\hp_{n,x}$ and are not necessarily in $\mathcal{N}^\hp$) such that
		$q=\sum_{i=1}^{I}\omega^i_{n,x}(q)\be^i_{n,x}(p)$ and
		\begin{equation}\label{eq:plip27}
		\lvert p-q\rvert = \sum_{i=1}^{I}\lvert \pi^i_{n,x}(p)-\omega^i_{n,x}\rvert\be_{n,x}^{i}(p).
		\end{equation}
		
		By the convexity of $V^\hp_n$, since $q=\sum_{i=1}^{I}\omega^i_{n,x}\be^i_{n,x}(p)$ it directly follows that
		$$
		V^\hp_n(x,\;q) \leq \sum_{i=1}^{I} V^\hp_n(x,\;\omega_{n,x}^i)\psi_{n,x}^{i}(p)\,.
		$$
		Using the above inequality, (\ref{eq:plip27}), the representation $ V^\hp_n(x,p)|_{K_p} = \sum_{i=1}^{I} V^\hp_n(x,\;\pi_{n,x}^i)\be_{n,x}^{i}(p)$ and
		$\vexp[f]\leq f$ we get on recalling (\ref{eq:algo4})
		that
		\begin{align}\notag
		0&\leq V^\hp_n(x,\;q)-V^\hp_n(x,\;p)\leq \sum_{i=1}^{I} \big(V^\hp_n(x,\;\omega_{n,x}^i)-V^\hp_n(x,\;\pi_{n,x}^i)\big)\be_{n,x}^{i}(p)
		\\
		\label{eq:plip19}
		\begin{split}
		&\leq \sum_{i=1}^{I}\Big(\EE\big[V^{\hp}_{n + 1}(\oX^{x}_{n+1},\;\omega_{n,x}^i)\big] + \hht H\big(\tn,\;x,\;Z^{\hp}_{n}(x,\;\omega_{n,x}^i),\;\omega_{n,x}^i\big)
		\\
		&\hspace{25pt}-\EE\big[V^{\hp}_{n + 1}(\oX^{x}_{n+1},\;\pi_{n,x}^i)\big] - \hht H\big(\tn,\;x,\;Z^{\hp}_{n}(x,\;\pi_{n,x}^i),\;\pi_{n,x}^i\big)\Big)\be_{n,x}^{i}(p)\, ,
		\end{split}
		\end{align}
		where we used that $V^{\hp}_{n}(x,\;\pi_{n,x}^i) = \vexp[\mathcal{Y}_{n}(x)](\pi_{n,x}^i)$, $i=1,\dots,I$.
		
		By the Lipschitz continuity of $V^{\hp}_{n + 1}$, it follows from \eqref{eq:plip19} using \eqref{eq:plip27} and \cite[Lemma\,3.6]{gruen2012aprobabilistic} that
		\begin{align}
		\label{eq:plip10}
		|V^\hp_n(x,\;q)-V^\hp_n(x,\;p)|\leq L_{n+1}\sum_{i=1}^{I}\lvert \pi_{n,x}^i-\omega_{n,x}^i\rvert\be_{n,x}^{i}(p)= L_{n+1}\lvert p-q\rvert,
		\end{align}
		where $l_\tau=L_{n+1}(1+ C\hht) + C\hht$.
		
		Recursively, we get that $l_\tau = l_\tau + Ct_n + C\hht\sum_{i=n+1}^{N}l_\tau$, $n =1,\ldots, N-1$ and
		by the discrete Gronwall lemma it follows $l_\tau \leq L_0\leq (l_\tau+CT)\exp(CT)$. 
		Hence $V^\hp_n$ is uniformly Lipschitz continuous in $p$ with a Lispchitz constant $L_0$ which only depends on the \cref{A1,A2,A3,A4}.
	\end{proof}
	
	\subsection{Lipschitz continuity in $x$}\label{subsec:xlip}
	The next lemma  can be shown analogically to \cite[Lemma~3.3]{gruen2012aprobabilistic}.
	\begin{lem}\label{lem:xlip0} Let $\phi:\RR^d\rightarrow\RR$ be a uniformly Lipschitz continuous function with a Lispschitz constant $L$. Then, there exists a constant $C>0$, depending only on the data of \cref{A1,A2,A3,A4}, such that the following inequality holds for $n=0,\dots,N-1$
		\begin{equation*}
		\begin{split}
		\Big\lvert\EE 
		&\big[ \phi(\oX_{n+1}^{x})\big] +\hht H\Big(\tn,\;x,\;\frac{1}{\hht}\EE \big[\phi(\oX_{n+1}^{x})\signTlx\xi_{n}\sqrt{\hht}\big],\;p\Big)
		\\
		&\hspace{10pt}-\EE\big[\phi(\oX_{n+1}^{x'})\big] -\hht H\Big(\tn,\;x',\;\frac{1}{\hht}\EE \big[\phi(\oX_{n+1}^{x'})\signTlxp\xi_{n}\sqrt{\hht}\big],\;p\Big)\Big\rvert\leq C_{\hht,L}\llvert x-x' \rrvert ,
		\end{split}
		\end{equation*}
		where $C_{\hht,L}:= L(1 + C\hht) + C\hht$.
	\end{lem}
	
	{
		\begin{lem}[Lipschitz continuity in $x$]\label{lem:xlip}
			For $n=0,\dots,N$ the interpolant $V_n^\hp$ is 
			\begin{assertion}
				\item Lipschitz continuous in $x$:
				\begin{equation*}
				\lvert V_n^\hp(x,\;p)-V_n^\hp(x',\;p)\rvert\leq C\lvert x-x'\rvert \quad\mbox{ for all }\quad x,x'\in\RR^d, \,\,p\in\Delta(I)\,,
				\end{equation*}
				\item uniformly bounded:
				\begin{equation*}
				\lvert V_n^\hp(x,\;p)\rvert\leq C\quad\mbox{ for all }\quad x\in\RR^d,\,\,p\in\Delta(I),
				\end{equation*}
			\end{assertion}
			where $C>0$ is a constant which depends only on \cref{A1,A2,A3,A4}.
		\end{lem}
		From the subsequent proof it follows that the non-convex interpolant $\widetilde{V}^\hp_n$ defined in (\ref{eq:pinter2})
		enjoys the same boundedness and Lipschitz continuity properties as $V_n^\hp$.
		\begin{proof}
			We fix $p\in\Delta(I)$ and consider $x,x'\in\RR^d$. For $n=N$ we have $\widetilde{V}_N^{\hp}(x,\;p) = V_N^{\hp}(x,\;p) =\la p,\;g(x)\ra $ and  $(i)$, $(ii)$ 
			hold since
			\begin{align*}
			\big\lvert V_N^{\hp}(x,\;p)-V_N^{\hp}(x',\;p)\big\rvert &= \big\lvert \la p,\;g(x)-g(x')\ra\big\rvert\leq L_{N}\lvert x-x'\rvert,
			\\
			\big\lvert V_N^{\hp}(x,\;p)\big\rvert &= \big\lvert \la p,\;g(x)\ra\big\rvert\leq C_{N},
			\end{align*}
			where $l_\tau$ and $C_N$ are positive constant which depend only on $g$.
			
			We proceed by induction. We assume that $V^{\hp}_{n + 1}(x,\;p)$, $\widetilde{V}^{\hp}_{n + 1}(x,\;p)$ are Lipschitz continuous in $x$ with a Lipschitz constant $L_{n + 1}$ and bounded by $C_{n+1}$.
			We show that $V^{\hp}_{n}(x,\;p)$, $\widetilde{V}^{\hp}_{n}(x,\;p)$ are Lipschitz continuous with a Lipschitz constant $l_\tau$ and bounded by a constant $C_{n}$. 
			On recalling (\ref{eq:algo5}), (\ref{vtilde}) we may write
			\begin{align}\notag
			V^{\hp}_{n}(x,p) &= \vexp\big[ \widetilde{V}^{\hp}_n(x,\cdot)\big](p) 
			\\
			\label{eq:xlip21}
			&= \vexp\bigg[ \EE\big[\widetilde{V}^{\hp}_{n + 1}(\oX^{x}_{n+1},\cdot)\big] - \hht H\big(\tn,\;x,\;\widetilde{Z}^{\hp}_{n}(x,\cdot),\cdot\big)\bigg](p),
			\end{align}
			where $\widetilde{Z}^{\hp}_n(x,\;p) := \frac{1}{{\hht}}\EE\big[\widetilde{V}^{\hp}_{n + 1}(\oX^{x}_{n + 1},\;p)\signTlx\xi_{n}\sqrt{\hht}\big]$.
			Moreover, we recall that for $p_m\in\mN^\hp$ it holds by definition
			\begin{equation}\label{vexp_vtilde}
			\begin{split}
			V^{\hp}_n(x,p_m) = \vexp\bigg[ \EE\big[\widetilde{V}^{\hp}_{n + 1}(\oX^{x}_{n+1},\cdot)\big] - \hht H\big(\tn,\;x,\;\widetilde{Z}^{\hp}_{n}(x,\cdot),\cdot\big)\bigg](p_m)\,.
			\end{split}
			\end{equation}
			
			\textit{i) Lipschitz continuity}. By \cref{lem:xlip0} we have for $p_m\in \mN^\hp$
			\begin{align}\label{eq:xlip23}
			\begin{split}
			&\EE\big[\widetilde{V}^{\hp}_{n + 1}(\oX^{x}_{n+1},p_m)\big] + \hht H\big(\tn,\;x,\;\widetilde{Z}^{\hp}_{n}(x,p_m),p_m\big)
			\\
			&\hspace{50pt}\leq\EE\big[\widetilde{V}^{\hp}_{n + 1}(\oX^{x'}_{n+1},p_m)\big] - \hht H\big(\tn,\;x',\;\widetilde{Z}^{\hp}_{n}(x',p_m),p_m\big) + L_{n+1}\lvert x-x'\rvert,
			\end{split}
			\end{align}
			with $L_{n+1}:=l_\tau(1 + C\tau) + C\tau$. 
			On noting \eqref{vexp_vtilde} it follows from
			\eqref{eq:xlip23} by \cref{prop:monotonicvexm} that
			\begin{align}\label{eq:xlip28}
			\widetilde{V}^\hp_n(x,p_m)-\widetilde{V}^\hp_n(x',p_m)\leq L_{n+1}\lvert x-x'\rvert
			\end{align}
			We recall (\ref{eq:pinter2}) and obtain from (\ref{eq:xlip28}) (note $\widetilde{V}^\hp_n(x,p_m)=V^{m}_n(x)$) that for any $p\in \Delta(I)$ it holds
			\begin{align}\label{lip_vtilde}
			\widetilde{V}^{\hp}_{n}(x,p) - \widetilde{V}^{\hp}_{n}(x',p) \leq \sum_{m=1}^{M}\Big(V^{m}_n(x) - V^{m}_n(x')\Big)\be^{m}(p)
			\leq L_{n+1}\lvert x-x'\rvert\,,
			\end{align}
			where we used that $\sum_{m=1}^{M}\be^{m}\equiv 1$, $\be^{m} \geq 0$ for any $0\leq n\leq N$.
			Consequently by (\ref{eq:xlip21}) and \cref{prop:monotonicvexm} it also follows that
			\begin{align}\label{lip_v}
			{V}^{\hp}_{n}(x,p) - {V}^{\hp}_{n}(x',p) \leq L_{n+1}\lvert x-x'\rvert\,.
			\end{align}

			After commuting the role of $x$ and $x'$ and repeating the above steps we obtain for any $p\in\Delta(I)$
			\begin{align}\label{eq:xlip29}
			\lvert V^\hp_n(x,p)-V^\hp_n(x',p)\rvert\leq L_{n+1}\lvert x-x'\rvert.
			\end{align}
			
			{Hence, we get recursively that $l_\tau \leq l_\tau + Ct_n + C\hht\sum_{i=n+1}^{N}L_i$.} By the discrete Gronwall lemma, it follows that $l_\tau\leq (l_\tau+CT)\exp(CT)$. 
			Consequently, $V^\hp_n$, $\widetilde{V}^\hp_n$, $n=0,\dots,N$ are Lipschitz continuous in $x$, with a Lipschitz constant $L:=(l_\tau+CT)\exp(CT)$ depending only on the data in \cref{A1,A2,A3,A4}.

			\textit{ii) boundedness}. Let $K_p = [\pi_{n,\;x}^1(p),\ldots,\pi_{n,\;x}^I(p)]$ be a simplex of $\mM_{n,\;x}^h$ 
			such that $p\in \overline{K}_p$, i.e. $p=\sum_{i=1}^{I}\pi^i_{n,\;x}(p)\psi^i_{n,\;x}(p)$, where $\{\psi_{n,\;x}^i(p):i=1,\ldots,I\}$ is the Lagrange polynomial basis on $K_p$. 
			In particular, $\sum_{i=1}^{I}\psi^i_{n,\;x}(p) = 1$ and $\psi^i_{n,\;x}(p)\geq 0 $ for $i=1,\ldots,I$. On recalling \eqref{eq:pinter} we may write
			\begin{equation}\label{eq:xlip16}
			\begin{split}
			V^{\hp}_{n}(x,\;p)
			= \sum_{i=1}^{I} &\Big(\EE\big[V^{\hp}_{n + 1}(\oX^{x}_{n+1},\;\pi^i_{n,\;x}(p))\big] 
			\\
			&\hspace{20pt}+ \hht H\big(\tn,\;x,\;Z^{\hp}_{n}(x,\;\pi^i_{n,\;x}(p)),\;\pi^i_{n,\;x}(p)\big)\Big)\psi_{n,\;x}^{i}(p),
			\end{split}
			\end{equation}
			where $Z^{\hp}_n(x,\;p) := \frac{1}{{\hht}}\EE\big[V^{\hp}_{n + 1}(\oX^{x}_{n + 1},\;p)\signTlx\xi_{n}\sqrt{\hht}\big]$. 
			By \eqref{eq:3}, since $V_{n+1}^\hp$ is bounded by $C_{n+1}$ we estimate the right hand side of \eqref{eq:xlip16} 
			\begin{align}
			\label{eq:xlip26}
			\begin{split}
			&\big\lvert\EE\big[V^{\hp}_{n + 1}(\oX^{x}_{n+1},\;\pi^i_{n,\;x}(p))\big] 
			\\
			&\hspace{20pt}+ \hht H\big(\tn,\;x,\;Z^{\hp}_{n}(x,\;\pi^i_{n,\;x}(p)),\;\pi^i_{n,\;x}(p)\big)\big\rvert\leq C_{n+1} + \hht C\big(1 + \lvert Z^{\hp}_{n}(x,\;\pi^i_{n,x}(p))\rvert\big).
			\end{split} 
			\end{align}
			Next, we show that $Z^{\hp}_{n}(x,\;\pi^i_{n,x}(p))$ is bounded. Since $V_{n+1}^\hp$ is uniformly Lipschitz continuous in the  variable $x$,
			on recalling \eqref{xeuler1},
			by the generalized mean value theorem \cite[Theorem 2.3.7 ]{clarke1990optimization} there exists $\Theta\in\RR^d$ with $|\Theta|_{\infty}\leq C$ such that
			\begin{equation}\label{eq:xlip22}
			V_{n + 1}^\hp(\oX_{n + 1}^{x},\;\pi^i_{n,\;x}(p)) = V_{n + 1}^\hp(x,\;\pi^i_{n,x}(p)) + \la\Theta,\;\sig_{n}(x)\xi_{n}\sqrt{\hht}\ra.
			\end{equation}
			We multiply \eqref{eq:xlip22} with $(1/\hht)\signTlx\xi_{n}$ and take the expectation to get
			\begin{align}
			\notag
			\begin{split}
			Z^{\hp}_{n}(x,\;\pi^i_{n,\;x}(p))&= \frac{1}{\hht}\EE\big[V_{n + 1}^\hp(\oX_{n + 1}^{x},\;\pi^i_{n,\;x}(p))\signTlx\xi_{n}\sqrt{\hht}\big] 
			\\
			&= \frac{1}{\hht}\EE\big[V_{n + 1}^\hp(x,\;\pi^i_{n,x}(p))\signTlx\xi_{n}\sqrt{\hht} + \la\Theta,\;\sig_{n}(x)\xi_{n}\sqrt{\hht}\ra\signTlx\xi_{n}\sqrt{\hht}\big]
			\end{split}
			\\
			\label{eq:xlip24}
			&=\frac{1}{\hht}\EE\big[\la\Theta,\;\sig_{n}(x)\xi_{n}\sqrt{\hht}\ra\signTlx\xi_{n}\sqrt{\hht}\big]. 
			\end{align}
			By \cref{A2} $\sig_{n}$ and $\signTl$ are uniformly bounded.  Hence, it follows from \eqref{eq:xlip24} that
			\begin{align}
			\label{eq:xlip25}
			\lvert Z^{\hp}_{n}(x,\;\pi^i_{n,x}(p))\rvert\leq \frac{1}{\hht}\|\sig_{n}\|_\infty\|\signTl\|_\infty\EE \big[|\Theta|_{\infty}\lvert\xi_{n}\sqrt{\hht}\rvert^2\big]\leq C.
			\end{align}
			We substitute \eqref{eq:xlip25}, \eqref{eq:xlip26} into \eqref{eq:xlip16} and get that
			\begin{align*}
			\lvert 	V^{\hp}_{n}(x,\;p)\vert\leq \sum_{i=1}^{I}C_n\psi_{n,\;x}^{i}(p)=C_n,
			\end{align*}
			where $C_{n}:=C_{n+1} + \hht C$. Consequently, $V^{\hp}_{n}(x,\;p)$, $n=0,\dots,N$ is uniformly bounded by $C:=C_N + CT$.
	\end{proof}}
	
	\subsection{Almost H{\"o}lder continuity in $t$}\label{subsec:thol}
	{
		\begin{lem}\label{lem:thol}
			For any $\hht>0$, $\hp>0$ and $x\in \RR^d$, $p\in\Delta(I)$ the interpolant $V_{\hht}^{\hp}$ defined in \eqref{eq:tinter} satisfies the following inequality
			\begin{equation*}
			\begin{split}
			\lvert V_{\hht}^{\hp}(s,x,\;p)-V_{\hht}^{\hp}(t,\;x,\;p)\rvert  
			\leq C\lvert s -t\rvert^{1/2} + C\hht^{1/2}\qquad \forall s,t\in[0,\;T],
			\end{split}
			\end{equation*}
			where the constant $C>0$ depends only on \cref{A1,A2,A3,A4}.
		\end{lem}
		\begin{proof}
			We consider the piecewise linear Lagrange basis functions associated with $\Pi_\hht$ as
			\begin{equation*}
			\chi_{n}(t)=
			\left\{
			\begin{array}{lllll}
			\D \frac{t - t_{n-1}}{\hht},\quad &\mbox{for}\,\, t\in[t_{n-1}, t_n]\,,\\
			\D \frac{t_{n+1}-t}{\hht},\quad& \mbox{for}\,\,t\in[t_{n}, t_{n+1}]\,,\\
			\D 0 \qquad & \mbox{otherwise}\,,
			\end{array}
			\right.
			\end{equation*}
			for $n=0,\dots,N$ and note that $\sum_{n=0}^{N}\chi_{n}(t) = 1$ for $t\in[0,T]$.
			
			We consider $s,t\in[0,\;T]$ where {$t\in[t_{n},\;t_{n+1}]$ and $s\in[t_{n + n'},\;t_{n + n' + 1}]$.
				We deduce from \eqref{eq:tinter} that}
			\begin{align}
			\notag
			\begin{split}
			V_{\hht}^{\hp}(t,\;x,p)-V_{\hht}^{\hp}(s,\;x,p)&=\sum_{k=0}^{1}V_{n+k}^{\hp}(x,p)\chi_{n+k}(t)-\sum_{k'=0}^{1}V_{n+n'+k'}^{\hp}(x,p)\chi_{n+n'+k'}(s)
			\end{split}
			\\
			\label{eq:thold0}
			&=\sum_{k=0}^{1}\sum_{k'=0}^{1}\big(V_{n+k}^{\hp}(x,p)-V_{n+n'+k'}^{\hp}(x,p)\big)\chi_{n+k}(t)\chi_{n+n'+k'}(s)\,.
			\end{align} 
			for $x\in \RR^d$ and $p\in\Delta(I)$.
			
			Since $\vexp[f]\leq f$ we get for $p_m\in\mN^\hp$, recall (\ref{eq:algo4}), (\ref{eq:algo5}), that
			\begin{align}\label{eq:thol1}
			V_{n + k}^{\hp}(x,p_m)\leq \EE\big[V^{\hp}_{n + k+ 1}(\oX^{n+k,\;x}_{n + k+ 1},p_m)\big] + \hht H\big(t_{n + k},\;x,\;{Z}^{\hp}_{n + k}(x,p_m),p_m\big),
			\end{align}
			with ${Z}^{\hp}_{n + k}(x,p_m) := \frac{1}{{\hht}}\EE\big[V^{\hp}_{n +k+ 1}(\oX^{n+k,\;x}_{n +k+ 1},p_m)\sig_{n+k}^{-T}\xi_{n + k}\sqrt{\hht}\big]$. 
			Using \eqref{eq:xlip25} and \eqref{eq:3} we obtain from \eqref{eq:thol1} that
			\begin{equation}\label{eq:sup26}
			\begin{split}
			&V_{n + k}^{\hp}(x,p_m)-V_{n+n' + k'}^{\hp}(x,p_m) 
			\\
			& \leq \EE\big[V^{\hp}_{n + k+ 1}(\oX^{n+k,\;x}_{n + k+ 1},p_m)-V_{n + n' + k'}^{\hp}(x,p_m)\big] + \hht H\big(t_{n + k},\;x,\;{Z}^{\hp}_{n + k}(x,p_m),p_m\big)
			\\
			& \leq \EE\big[V^{\hp}_{n +k+ 1}(\oX^{n+k,\;x}_{n+k+1},p_m)-V_{n + n' + k'}^{\hp}(x,p_m)\big] + C\hht\big( 1 + C\lvert {Z}^{\hp}_{n + k}(x,p_m)\rvert\big)
			\\
			&  \leq \EE\big[V^{\hp}_{n +k+ 1}(\oX^{n+k,\;x}_{n+k+1},p_m)\big]-V_{n + n' + k'}^{\hp}(x,p_m) + C\hht.
			\end{split}
			\end{equation}
			Recursively, we estimate the first term on the right-hand side above using the corresponding analogue of \eqref{eq:thol1} as
			\begin{align}
			\notag
			& V_{n + k + 1}^{\hp}(\oX^{n+k,\;x}_{n+k+1},p_m)-V_{n+n' + k'}^{\hp}(x,p_m)
			\\
			\label{eq:sup22}
			&\hspace{10pt}\leq \EE\Big[V^{\hp}_{n +k+ 2}(\oX^{n+k+1,\;\oX^{n+k,\;x}_{n+k+1}}_{n+k+2},p_m)\Big]-V_{n + n' + k'}^{\hp}(x,p_m) +  C\hht.
			\end{align} 
			We substitute \eqref{eq:sup22} into \eqref{eq:sup26} and obtain on noting $\oX^{n+k,\;x}_{n+k+2}= \oX^{n+k+1,\;\oX^{n+k,\;x}_{n+k+1}}_{n+k+2} $ (cf. \eqref{eq:Xeuler}) that
			\begin{equation*}
			V_{n + k}^{\hp}(x,p_m)-V_{n+n' + k'}^{\hp}(x,p_m)\leq \EE\Big[V^{\hp}_{n +k+ 2}(\oX^{n+k,\;x}_{n+k+2},p_m)-V_{n + n' + k'}^{\hp}(x,p_m)\Big] +  C2\hht.
			\end{equation*}
			Consequently, we get after $(n' + k'-k -2)$ recursive steps that
			\begin{equation}
			\label{eq:thol2}
			\begin{split}
			&V_{n + k}^{\hp}(x,p_m)-V_{n+n' + k'}^{\hp}(x,p_m)
			\\
			&\hspace{20pt}\leq \big\lvert\EE\big[V^{\hp}_{n +n'+ k'}(\oX^{n+k,\;x}_{n+n' + k'},p_m)-V^{\hp}_{n + n' +k'}(x,p_m) \big]\big\rvert + (n'+ k'-k)C\hht.
			\end{split}
			\end{equation}
			By \cref{lem:xlip} and \cref{A2} we estimate the first term on the right hand side of \eqref{eq:thol2} as
			\begin{equation}
			\nonumber 
			\begin{split}
			&\big\lvert\EE\big[ V^{\hp}_{n+n' +k'}(x,p_m) - V^{\hp}_{n + n' + k'}(\oX^{n+k,x}_{n+n' +k'},p_m)\big]\big\rvert\leq C \big\lvert\EE\big[ x -\oX^{n+k,x}_{n+n' +k'}\big]\big\rvert
			\\
			&\hspace{20pt} 
			\leq C\bigg[\sum_{j=n+ k}^{n+n'+ k'-1}\EE \llvert \sigma_j(\oX^{n+k,\;x}_{j})\rrvert^2\hht\bigg]^{1/2}\leq C\llvert t_{n+ k} - t_{n+n'+ k'}\rrvert^{1/2}\,,
			\end{split}
			\end{equation}
			and get
			\begin{align}\label{lt0}
			&V_{n+ k}^{\hp}(x,p_m)-V_{n+n'+ k'}^{\hp}(x,p_m)\leq C\llvert t_{n+ k} - t_{n+n'+ k'}\rrvert^{1/2} + C(t_{n+n' +k'} -t_{n+ k})\,.
			\end{align} 
			Since  $t\in[t_{n},\;t_{n + 1}]$ and $s\in[t_{n + n'},\;t_{n + n' + 1}]$ (i.e.\ $|t-t_{n+k}|< \tau$, $|s-t_{n+n'+k'}|< \tau$ for $k,\, k'=0,1$), we have from (\ref{lt0}) by the triangle inequality, and the fact that $|t-s| \leq T^{1/2}|t-s|^{1/2}$,
			\begin{equation}\label{eq:thol5}
			V_{n+ k}^{\hp}(x,p_m)-V_{n+n'+ k'}^{\hp}(x,p_m)\leq C T^{1/2}\lvert t-s\vert^{1/2} + C\hht^{1/2},
			\end{equation}
			for $k,\, k'=0,1$ and $p_m\in\mN^\hp$.
			
			{
				On recalling that $V_{n}^{\hp}(x,p_m)=\widetilde{V}_{n}^{\hp}(x,p_m)$ for $p_m\in\mN^\hp$, $V_{n}^{\hp}(x,p)=\vexp\big[\widetilde{V}_{n}^{\hp}(x,\cdot)\big](p)$ for $p\in\Delta(I)$,
				we deduce analogically as in the proof of Lemma~\ref{lem:xlip} (cf. \eqref{lip_vtilde}, \eqref{lip_v}),
				that the inequality (\ref{eq:thol5}) holds for any $p\in \Delta(I)$ .}
			Hence, substituting \eqref{eq:thol5} into \eqref{eq:thold0} for $p\in \Delta(I)$ implies the inequality
			\begin{equation}
			V_{\tau}^{\hp}(t,x,p)-V_{\tau}^{\hp}(s,x,p)\leq   C\lvert t-s\vert^{1/2} + C\hht^{1/2}\,.
			\end{equation}
			After reverting the role of $s$ and $t$ and repeating the above steps we get the statement of the lemma.
		\end{proof}
	}
	
	\subsection{Monotonicity}\label{subsec:monot}
	{
		\begin{lem}\label{lem:monotonicity}
			Let $\phi_1,\;\phi_2:\RR^d\rightarrow \RR$ be two uniformly Lipschitz continuous functions that satisfy $0\leq (\phi_1-\phi_2)\leq C$. Then for any $x\in\RR^d$, $p\in\Delta(I)$, $\hht>0$,  $n= 0,\dots,N-1$ it holds that
			\begin{equation*}
			\begin{split}
			\EE&\big[\phi_1(\oX_{n+1}^x)\big] + \hht H\Big(t_n,\; x,\;\frac{1}{\hht}\EE\big[\phi_1(\oX_{n+1}^x)\signTlx\xi_{n}\sqrt{\hht} \big],\;p \Big)
			\\
			&\hspace{50pt}\geq \EE\big[\phi_2(\oX_{n+1}^x)\big] + \hht H\Big(t_n,\;x,\;\frac{1}{\hht}\EE\big[\phi_2(\oX_{n+1}^x)\signTlx\xi_{n}\sqrt{\hht} \big],\;p \Big) -C\hht\sqrt{\hht},
			\end{split}
			\end{equation*}
			where $C>0$ is a constant which depends only on \cref{A1,A2,A3,A4}.
		\end{lem}
		\begin{proof}
			We set 
			\begin{equation*}
			\begin{split}
			\mH:= \EE\big[(\phi_1-\phi_2)(\oX_{n+1}^x)\big] + \hht &H\Big(t_n,\; x,\;\frac{1}{\hht}\EE\left[\phi_1(\oX_{n+1}^x)\signTlx\xi_{n}\sqrt{\hht}\right],\;p \Big) 
			\\
			&-\hht H\Big(t_n,\;x,\;\frac{1}{\hht}\EE\left[\phi_2(\oX_{n+1}^x)\signTlx\xi_{n}\sqrt{\hht} \right],\;p \Big).
			\end{split}
			\end{equation*}
			By \eqref{eq:4} and the generalized mean value theorem \cite[Theorem 2.3.7 ]{clarke1990optimization} there exists a $\Theta\in \RR^d$, $|\Theta|_{\infty}\leq C$ such that
			\begin{align}
			\notag
			\begin{split}
			\mH&= \EE\big[(\phi_1-\phi_2)(\oX_{n+1}^x)\big (1 + \big \la \Theta, \signTlx\xi_{n}\sqrt{\hht}\big\ra\big)\Big]
			\end{split}
			\\
			\label{eq:monot1}
			\begin{split}
			&= \EE\big[\ind_{\big\{C\lVert\sig^{-1}\rVert_{\infty}\lvert \xi_n\rvert\sqrt{\hht}\geq 1\big\}}(\phi_1-\phi_2)(\oX_{n+1}^x)\big (1 + \big \la \Theta, \;\signTlx\xi_{n}\sqrt{\hht}\big\ra\big)\Big]
			\\
			&\hspace{25pt}+ \EE\big[\ind_{\big\{C\lVert\sig^{-1}\rVert_{\infty}\lvert \xi_n\rvert\sqrt{\hht}< 1\big\}}(\phi_1-\phi_2)(\oX_{n+1}^x)\big (1 + \big \la \Theta, \;\signTlx\xi_{n}\sqrt{\hht}\big\ra\big)\Big].
			\end{split}
			\end{align}
			Next, we show that the second term of the right hand side of \eqref{eq:monot1} is positive. Since $(\phi_1-\phi_2)\geq 0$, it remains to examine the term $(1 + \big \la \Theta,\; \signTlx\xi_{n}\sqrt{\hht}\big\ra\big)$. 
			We note that
			\begin{align}
			\nonumber
			1 + \big \la \Theta,\; \signTlx\xi_{n}\sqrt{\hht}\big\ra&\geq 1 -\|\Theta\sig_{n}^{-1}\|_\infty \lvert \xi_{n}\rvert\sqrt{\hht}\geq 1 -C\lVert\sig^{-1}_n\rVert_{\infty}\lvert \xi_{n}\rvert\sqrt{\hht}.
			\end{align}
			For $C\lVert \sig^{-1}\rVert_{\infty}\lvert \xi_{n}\rvert\sqrt{\hht}<1$ it holds  $\big(1 + \big \la \Theta,\; \signTlx\xi_{n}\sqrt{\hht}\big\ra\big)> 0$ and hence
			\begin{equation}\label{eq:monot3}
			\EE\big[\ind_{\big\{C\lVert\sig^{-1}\rVert_{\infty}\lvert \xi_n\rvert\sqrt{\hht}< 1\big\}}(\phi_1-\phi_2)(\oX_{n+1}^x)\big (1 + \big \la \Theta, \;\signTlx\xi_{n}\sqrt{\hht}\big\ra\big)\Big]\geq 0.
			\end{equation}
			Using\eqref{eq:monot3} we deduce from \eqref{eq:monot1} that
			\begin{equation*}
			\begin{split}
			\mH\geq \EE\Big[\ind_{\big\{\llvert \xi_{n}\sqrt{\hht} \rrvert^2\geq 1/R\big\}}\big\langle \Theta,\;  (\phi_1-\phi_2)(\oX_{n+1}^x)\signTlx\xi_{n}\sqrt{\hht} \big\rangle\Big],
			\end{split}
			\end{equation*}
			where $R:=C^2\lVert \sig^{-1}\rVert^2_{\infty}$. 
			
			On noting that $\lvert \xi_n\rvert = \lvert \xi_n^{1}\rvert + \ldots + \lvert \xi_n^{d}\rvert = (1 + \ldots + 1) = d$ we deduce
			\begin{align*}
			&\EE\Big[\ind_{\big\{\llvert \xi_{n} \rrvert^2{\hht}\geq 1/R\big\}}\lvert\xi_{n}\rvert\sqrt{\hht}\Big] 
			\\
			&= \ind_{\big\{d^2\hht\geq 1/R\big\}}d\sqrt{\hht}
			=
			\ind_{\big\{d^2\hht\geq 1/R\big\}}R\frac{1}{R}d\sqrt{\hht}
			\leq \ind_{\big\{d^2\hht\geq 1/R\big\}}R\hht d^3\sqrt{\hht}\leq R\hht d^3\sqrt{\hht}.
			\end{align*}
			
			Since $-(\phi_1-\phi_2)\geq -C$  we conclude
			\begin{equation*}
			\begin{split}
			\mH& \geq  - \EE\Big[\ind_{\big\{\llvert \xi_{n}\sqrt{\hht} \rrvert^2\geq 1/R\big\}}(\phi_1-\phi_2)(\oX_{n+1}^x)\lVert\Theta\sig^{-1}\rVert_{\infty} \llvert\xi_{n}\sqrt{\hht}\rrvert\Big]
			\\
			&\geq - C \EE\Big[\ind_{\big\{\llvert \xi_{n}\sqrt{\hht} \rrvert^2\geq 1/R\big\}}\lvert\xi_{n}\sqrt{\hht}\rvert\Big]
			= - CRd^3\hht \sqrt{\hht}.
			\end{split}
			\end{equation*}	
	\end{proof}}
	
	\section{Convergence of the numerical approximation}\label{sec_conv}
	
	In this section we prove the convergence of numerical approximation, see \cref{thm:1} below, in several steps. 
	First, we show that, up to a subsequence, the sequence $\{V^\hp_\hht\}_{\hp,\;\hht>0}$ admits a limit denoted by $w$.
	We then prove the viscosity super/sub-solution property of every accumulation point $w$. Hence, by the uniqueness property of the viscosity solution, see \cite{cardaliaguet2009a},
	we may conclude that the whole sequence $\{V^\hp_\hht\}_{\hp,\;\hht>0}$ converges to the viscosity solution.
	
	\begin{theorem}\label{thm:1}
		Under \cref{A1,A2,A3,A4} the numerical solution $V^{\hp}_{\hht}$ converges to the viscosity solution of \eqref{eq:1} (uniformly on compact subsets of $[0,\;T]\times\RR^d\times\Delta(I)$) in the sense that
		for all $(t',\;x',\;p')\to(t,\;x,\;p)$ it holds that
		\begin{equation*}
		\lim_ {\hht, \hp\rightarrow 0} V^{\hp}_{\hht}(t',x',p') = V(t,\;x,\;p),
		\end{equation*}
		where $V$ is the unique uniformly bounded and continuous viscosity solution to \eqref{eq:1} which is convex and uniformly Lipschitz continuous in $p$. 
	\end{theorem}
	
	\subsection{Existence of a limit}\label{ssec:41}
	\begin{lem}\label{lem:5}
		The sequence $\{ V_\hht^\hp\}_{\hht,\;\hp>0}$ admits a subsequence
		which converges uniformly on every compact subset of $[0,\;T]\times\RR^d\times\Delta(I)$ to a uniformly bounded and continuous function $w$ which is convex and uniformly Lipschitz continuous in $p$.
	\end{lem}
	
	\begin{proof}
		The proof is a consequence of a slight modification of  the Arzel\`a--Ascoli theorem, \cite[Section~III.3]{yosida2013functional}. The equi-boundedness is granted by \cref{lem:xlip} and the equi-continuity is granted by \cref{lem:plip0,lem:xlip,lem:thol}.
	\end{proof}
	
	\subsection{Viscosity solution properties and uniqueness of the limit}\label{susbec:42}
	
	Below, we show that every accumulation point $w$ of $\{V^\hp_\hht\}_{\hp,\;\hht>0}$ from \cref{lem:5} is a viscosity sub- and super-solution to \eqref{eq:1}
	which by uniqueness of the viscosity solution $V$ implies Theorem~\ref{thm:1}.
	
	\subsubsection{Viscosity subsolution property of $w$}
	
	\begin{prop}\label{prop:sub} Every accumulation point $w$ of the sequence $\{ V_\hht^\hp\}_{\hht,\;\hp>0}$
		is a viscosity subsolution of \eqref{eq:1} on $[0,\;T]\times\RR^d\times\Delta(I)$.
	\end{prop}	
	\begin{proof} Let $\phi:[0,\;T]\times\RR^d\times\Delta(I)\rightarrow\RR$ be a smooth test function such that $w-\phi $ has a strict global maximum at $(\ot,\;\ox,\;\op)$, where $\op\in\Delta(I)$.
		We have to show, that $\phi$ satisfies \eqref{eq:13} at $(\ot,\;\ox,\;\op)$.

		As a limit of convex functions $w$ is convex in $p$ and (cf. \cite[Theorem~1]{oberman2007the}) we have since $\op\in\Delta(I)$
		\begin{equation}\label{eq:15}
		\lbdmin\left(\op,\;\frac{\partial^2\phi}{\partial p ^2}(\ot,\;\ox,\;\op)\right)\geq 0.
		\end{equation} 
		
		Similarly to \cite[Lemma~2.4]{bardi2009optimal} we note that there exists a sequence $(\ott,\;\oxt,\;\opt)_\hht$ such that $\ott =  l_\tau\hht \in\Pi_\hht$, $ l_\tau\in\mathbb{N}$ converges to $\ot$, $\opt\in\mN^{h}$ converges to $\op$, and $\oxt$ to $ \ox$
		for $\tau,h \rightarrow 0$; also $ V^\hp_\hht -\phi$ has a global maximum at $(\ott,\;\oxt,\;\opt)$ {on $\Pi_\hht \times \mathbb{R}^d\times \mN^{h}$}.
		
		Define $\phi^\hp_\hht := \phi  + (V_\hht^\hp-\phi)(\ott,\;\oxt,\;\opt)$. Then for all $x\in\RR^d,\;p_m\in\mN^\hp$ we have
		\begin{equation}\label{eq:16}
		V_\hht^\hp(\ott + \hht,\;x,\;p_m) -\phi_\hht^\hp(\ott + \hht, \;x, \;p_m)\leq V_\hht^\hp(\ott,\;\oxt,\;\opt) -\phi_\hht^\hp(\ott,\;\oxt,\;\opt) =0.
		\end{equation}
		
		Throughout the proof we set
		\begin{equation*}
		\oX_{n + 1} := \oxt + \sigma(\ott,\;\oxt)\xi_{l_\tau}\sqrt{\hht};
		\end{equation*}
		and for each $m\in\big\{1,\ldots,M\big\}$:
		\begin{align}
		\notag
		\oZ^m_{n}(\oxt) 
		&:=\frac{1}{{\hht}}\EE\big[V_{\hht}^{\hp}(\ott + \hht,\;\oX_{n + 1}, \;p_m)(\sigma(\ott,\;\oxt))^{-T}\xi_{l_\tau}\sqrt{\hht}\big],
		\\
		\label{eq:sub3}
		\oY^m_{n}(\oxt) &:=\EE\big[V^\hp_{\hht}(\ott + \hht,\;\oX_{n+1},\;p_m)\big] + \hht H\big(\ott,\;\oxt,\;\oZ^m_{n}(\oxt),\;p_m\big).
		\end{align}
		We denote the non convex data set as $\mY_{n}(\oxt) := \big\{\oY^1_{n}(\oxt), \ldots, \oY^M_{n}(\oxt)\big\}$.  By definition $\vexp[f](p_m)\leq f(p_m)$. Thus from \eqref{eq:sub3} we have 
		\begin{align}
		\notag
		V^{\hp}_{\hht}(\ott,\;\oxt,\;\opt) &=\vexp\big[\mY_{n}(\oxt)\big](\opt)
		\\
		\label{eq:sub6}
		\begin{split}
		&\leq \EE\big[V^\hp_{\hht}(\ott + \hht,\;\oX_{n+1},\;\opt)\big] 
		\\
		&\hspace{20pt}+ \hht H\Big(\ott,\;\oxt,\;\frac{1}{{\hht}}\EE\big[V_{\hht}^{\hp}(\ott + \hht,\;\oX_{n + 1}, \;\opt)(\sigma(\ott,\;\oxt))^{-T}\xi_{l_\tau}\sqrt{\hht}\big],\;\opt\Big).
		\end{split}
		\end{align}
		
		We use \cref{lem:monotonicity} with $\phi_2(\cdot) := V^\hp_\hht(\ott+  \hht,\;\cdot,\;\opt)$ and $\phi_1(\cdot) := \phi^\hp_\hht(\ott+  \hht,\;\cdot,\;\opt)$. Then, by \eqref{eq:16} and \eqref{eq:sub6} it follows immediately that
		\begin{align}
		\notag
		\begin{split}
		0&\leq  \EE\big[V_\hht^\hp(\ott + \hht,\;\oX_{n+1},\;\opt)\big]-V_\hht^\hp(\ott,\;\oxt,\;\opt) \\
		&\hspace{20pt}+\hht H\Big(\ott,\;\oxt,\;\frac{1}{{\hht}}\EE\big[V_{\hht}^{\hp}(\ott + \hht,\;\oX_{n + 1}, \;\opt)(\sigma(\ott,\;\oxt))^{-T}\xi_{l_\tau}\sqrt{\hht}\big],\;\opt\Big) 
		\end{split}
		\\
		\label{eq:sub5}
		\begin{split}
		&\leq \EE\big[ \phi(\ott + \hht,\;\oX_{n+1},\;\opt)\big]- \phi(\ott,\;\oxt,\;\opt)
		\\
		&\hspace{20pt}+ \hht H\Big(\ott,\;\oxt,\;\frac{1}{ {\hht}}\EE\big[\phi_\hht^\hp(\ott + \hht,\;\oX_{n+1},\;\opt)(\sigma(\ott,\;\oxt))^{-T}\xi_{l_\tau}\sqrt{\hht} \big],\;\opt\Big) + C\hht\sqrt{\hht}.
		\end{split}
		\end{align}
		
		First, we calculate the expectation in the first term of the right hand side of \eqref{eq:sub5}. By the Taylor expansion
		\begin{equation}\label{eq:26}
		\begin{split}
		\phi&(\ott +\hht,\;\oX_{n+1},\;\opt) 
		\\
		&\hspace{5pt}=\phi(\ott,\;\oxt,\;\opt)+\Big[\dt\phi(\ott,\;\oxt,\;\opt)+\frac12\Tr\big(\sigma\sigma^T(\ott,\;\oxt)D^2_x \phi(\ott,\;\oxt,\;\opt)\big)\Big]\hht
		\\
		&\hphantom{=\phi(\ott,\;\oxt,\;\opt)+}+ \big[D_x\phi(\ott,\;\oxt,\;\opt)
		\big] \sigma(\ott,\;\oxt)\xi_{l_\tau}\sqrt{\hht} + \hht \mO(\hht) + \hht \mO(\hht^{1/2}),
		\end{split}
		\end{equation}
		where $\mO(\hht)\to 0 $ when $\hht\to 0$. Taking the expectation of \eqref{eq:26}, we obtain
		\begin{equation}\label{eq:27}
		\begin{split}
		\EE &\big[\phi(\ott +\hht,\;\oX_{n+1},\;\opt)\big]
		\\ 
		&=\phi(\ott,\;\oxt,\;\opt) + \Big[\dt\phi(\ott,\;\oxt,\;\opt)
		+\frac12\Tr\big(\sigma\sigma^T(\ott,\;\oxt)D^2_x \phi(\ott,\;\oxt,\;\opt)\big)\Big]\hht 
		\\
		&\hspace{240pt}+ \hht \mO(\hht) + \hht \mO(\hht^{1/2}).
		\end{split}
		\end{equation}
		
		Next, we calculate the expectation in the third term of the right hand side of \eqref{eq:sub5}.
		We multiply \eqref{eq:26} with $(\sigma(\ott,\;\oxt))^{-T}\xi_{l_\tau}\sqrt{\hht}$ to get
		\begin{equation}\label{eq:28}
		\begin{split}
		\phi(\ott +\hht,\;\oX_{n+1},\;\opt)&(\sigma(\ott,\;\oxt))^{-T}\xi_{l_\tau}\sqrt{\hht}
		=\phi(\ott, \;\oxt,\;\opt)(\sigma(\ott,\;\oxt))^{-T}\xi_{l_\tau}\sqrt{\hht}
		\\[5pt]
		&+ D_x\phi(\ott,\;\oxt,\;\opt)
		\sigma(\ott,\;\oxt)\xi_{l_\tau}\sqrt{\hht}(\sigma(\ott,\;\oxt))^{-T}\xi_{l_\tau}\sqrt{\hht} 
		\\[7pt]
		&+\dt\phi(\ott,\;\oxt,\;\opt)\hht (\sigma(\ott,\;\oxt))^{-T}\xi_{l_\tau}\sqrt{\hht}
		\\[5pt]
		&+\frac12\Tr\big(\sigma\sigma^T(\ott,\;\oxt)D^2_x \phi(\ott,\;\oxt,\;\opt)\big)\hht (\sigma(\ott,\;\oxt))^{-T}\xi_{l_\tau}\sqrt{\hht}
		\\[5pt]
		&+(\sigma(\ott,\;\oxt))^{-T}\xi_{l_\tau}\hht\sqrt{\hht} \mO(\hht) +(\sigma(\ott,\;\oxt))^{-T}\xi_{l_\tau}\hht\sqrt{\hht} \mO(\hht^{1/2}),
		\end{split}
		\end{equation}
		where $\mO(r)\to 0$ when $r\to 0$. 
		Taking the expectation in \eqref{eq:28} yields to
		\begin{equation}\label{eq:29}
		\frac{1}{{\hht}}\EE\big[\phi(\ott +\hht,\;\oX_{n+1},\;\opt)(\sigma(\ott,\;\oxt))^{-T}\xi_{l_\tau}\sqrt{\hht}\big]
		=
		D_x\phi(\ott,\;\oxt,\;\opt).
		\end{equation}
		By substituting \eqref{eq:27} and \eqref{eq:29} into \eqref{eq:sub5} we arrive to
		\begin{align*}
		0\leq
		\dt\phi(\ott,\;\oxt,\;\opt)
		+\frac12\Tr&\big(\sigma\sigma^T(\ott,\;\oxt)D^2_x \phi(\ott,\;\oxt,\;\opt)\big)
		\\
		&+ H(\ott,\;\oxt,\;D_x\phi(\ott,\;\oxt,\;\opt) ,\;\opt) \\
		&+ \hht \mO(\hht) 
		+ \hht \mO(\hht^{1/2}) + C\hht\sqrt{\hht}.
		\end{align*}
		Taking the limit $\hht,\,\hp\to0$ we get on recalling $(\ott,\;\oxt,\;\opt) \to (\ot,\;\ox,\;\op)$ that
		\begin{align}\label{eq:30}
		0\leq
		\dt\phi(\ot,\;\ox,\;\op)
		+\frac12\Tr\big(\sigma\sigma^T(\ot,\;\ox)D^2_x \phi(\ot,\;\ox,\;\op)\big)
		+ H\big(\ot,\;\ox,\;D_x\phi(\ot,\;\ox,\;\op),\;\op\big).
		\end{align}
		With \eqref{eq:15} and \eqref{eq:30}, we conclude that the limit $w$ of the sequence $\{V_\hht^\hp\}_{\hp,\;\hht>0}$ satisfies \eqref{eq:13}. Hence, $w$ is a viscosity subsolution of \eqref{eq:1}.
	\end{proof}
	
	\subsubsection{Viscosity supersolution property of $w$}\label{ssec:42}
	
	To establish the viscosity supersolution property of the limits of the numerical approximation in \cref{prop:sup} below,
	we construct in \cref{defn:onestepfb} martingale processes that satisfy a one step dynamic programming principle \cref{lem:dpp}, cf. \cite{gruen2012aprobabilistic}.
	
	For  $n=N-1,\ldots,0$, $x\in\RR^d$ and a given $p\in\overline{\Delta(I)}$, 
	we denote by
	$K_{n,\;x}(p) =[\pi^1_{n,\;x}(p),\ldots,\pi^I_{n,\;x}(p)]$ the simplex in $\mM^\hp_{n,\;x}$ such that $p\in \overline{K}_{n,\;x}(p)$, 
	and denote by $\{\psi^i_{n,\;x}(p):i=1,\ldots,I\}$ the Lagrange polynomial basis on $K_{n,\;x}(p)$. 
	By \eqref{eq:vex0} and \eqref{eq:pinter}, we write 
	\begin{equation}\label{eq:martingale0}
	\begin{split}
	&V_n^\hp(x,\;p)=\vexp\big[
	\oY^{1}_{n}(x),\ldots,\oY^{M}_{n}(x)\big](p)
	\\
	& = \sum_{i=1}^I \Big(\EE\big[V^{\hp}_{n+1}(\oX^x_{n+1},\;\pi_{n,x}^i(p))\big] 
	+ \hht H\big(t_n,\;x,\;Z^h_n(x,\;\pi^i_{n,\;x}(p)),\;\pi_{n,x}^i(p)\big)\Big)\psi^i_{n,\;x}(p),
	\end{split}
	\end{equation}
	with $Z^h_n(x,\;\pi^i_{n,\;x}(p)):=\frac{1}{{\hht}}\EE\big[V_{n+1}^{\hp}(\oX^x_{n + 1}, \;\pi_{n,x}^i(p))(\sigma_n^{-T}(x)\xi_{n}\sqrt{\hht}\big]$. The set of vertices of the triangle $K_{n,\;x}(p)$ will be denoted by $\mN^h_{n,\;x}(p):=\{\pi^1_{n,\;x}(p),\ldots,\pi^I_{n,\;x}(p)\}$.
	{
		\begin{defn}[One-step feedback]\label{defn:onestepfb}
			Let  $n\in \{N-1,\ldots,0\}$, $x\in \RR^d$, and $p = (p_1,\ldots,p_I)\in\Delta(I)$.  
			We define the one-step feedbacks $\fp_{n+1}^{i,\;x,\;p}$ as {$\mN^h_{n,\;x}(p)$-valued} random variables which are independent of {$\{\xi_n\}_{n=0}^{N-1}$} such that
			\begin{enumerate}[label = $\roman*)$]
				\item for $n =0,\ldots,N-1$
				\begin{enumerate}[label = $\alph*)$]
					\item if $p_{i} = 0$ set $\fp_{n+1}^{i,\;x,\;p} = p$
					\item if $p_{i}>0$: choose $\fp_{n+1}^{i,\;x,\;p}$ among $\big\{\pi_{n,x}^1(p),\ldots,\pi_{n,x}^I(p)\big\}\in \mathcal{N}^h_{n,x}$ with probability
					\begin{equation}\label{eq:martingale}
					\begin{split}
					&\PP\Big[ \fp_{n+1}^{i,\;x,\;p}= \pi_{n,\;x}^j(p)\Big|\big(\fp_{n'}^{i',\;x',\;p_{m'}}\big)_{i'\in\{1,\ldots,I\},\;x'\in\RR^d,\;m'\in\{1,\ldots,M\},\;n'\in\{1,\ldots,n\}}\Big] 
					\\
					&\hspace{50pt}= \frac{(\pi_{n,\;x}^j(p))_i}{p_i}\psi^j_{n,\;x}(p)
					\end{split}
					\end{equation}
				\end{enumerate}
				\item for $n=N$ set $\fp_{n+1}^{i,\;x,\;p} = e_i$, where $\{e_i:i=1,\ldots,I\}$ is the canonical basis of $\RR^{I}$.
			\end{enumerate}
			Furthermore we define $\fp_{n+1}^{x,\;p}:=\fp_{n+1}^{\bi,\;x,\;p}$, where the index $\bi$ is a random variable with law $p = (p_1,\ldots,p_I)$ (i.e. $\PP\big[\bi=i\big] = p_i$, $i=1,\ldots,I$), 
			independent of {$\{\xi_n\}_{n=0}^{N-1}$} and of the process  $\left(\fp_{n}^{j,\;x,\;p}\right)_{j\in\{1,\ldots,I\},\;x\;\in\RR^d,\;p\;\in\Delta(I),\;n\;\in\{1,\ldots,L\}}$.
		\end{defn}
		
		\begin{rem}
			The probability $p_i$ in \cref{defn:onestepfb} is the $i$-th component of the probability vector $p\in\Delta(I)$, i.e., 
			it is the probability of the ``chosen'' game. 
			In this case, the optimal behavior of Player~1 at time $t_{n+1}$ is derived from the one step feedback $\fp_{n+1}^{x,\;p}:=\fp_{n+1}^{\bi,\;x,\;p}$. 
			This feedback is the discrete version (in time and in $p$) of its continuous counterpart see \cite[Lemma\,3.2.]{cardaliaguet2009on}, \cite[Definition\,3.9]{gruen2012aprobabilistic}
			and \cite[Section~4.1]{cardal09num} for more details.
		\end{rem}
		
		The one-step feedback $\fp_{n+1}^{x,\;p}$ is a martingale and provides a representation formula for the discrete lower convex envelope.
		\begin{lem}\label{lem:dpp}
			For all $n = 0,\ldots,N-1,\;x\in\RR^d,\;p_m\in\mN^h$ we have
			\begin{equation*}
			\begin{split}
			V_n^h(x, p_m) = V_n^m(x)&=\vexp\big[\oY^{1}_{n}(x),\ldots,\oY^{M}_{n}(x)\big](p_m) 
			\\
			& = \EE \Big[V^\hp_{n+1}(\oX^x_{n+1},\;\fp_{n+1}^{x,\;p_m})+ \hht H\big(t_n,\;x,\;Z^\hp_n(x,\;\fp_{n+1}^{x,\;p_m} ),\;\fp_{n+1}^{x,\;p_m}\big)\Big],
			\end{split}
			\end{equation*}
			with $Z^\hp_n(x,\;\fp_{n+1}^{x,\;p_m} ):=\frac{1}{{\hht}}\EE\big[V_{n+1}^{\hp}(\oX^x_{n + 1}, \;\fp_{n+1}^{x,\;p_m})(\sigma_n^{-T}(x)\xi_{n}\sqrt{\hht}\big]$.
		\end{lem}
		\begin{proof} 
			We consider $n=0,\ldots,N-1$, $x\in\RR^d$, $p_m\in\mN^h$
			and note that the law of the process $\fp_{n+1}^{x,\;p_m}$ is given by \eqref{eq:martingale} and  $\fp_{n+1}^{x,\;p_m}$ is independent of $\bi$. 
			Hence it holds for any $f:\mathcal{N}^h\rightarrow \mathbb{R}$ that
			\begin{align}\label{frep}
			\nonumber
			\EE\big[f(\fp_{n+1}^{x,\;p_m})\big] &=\sum_{i=1}^{I}\EE\big[\ind_{\{\bi=i\}}f(\fp_{n+1}^{i,\;x,\;p_m})\big] =\sum_{i=1}^{I}\EE\big[\ind_{\{\bi=i\}}\big]\EE\big[f(\fp_{n+1}^{i,\;x,\;p_m})\big]
			\\
			\nonumber
			&=\sum_{i=1}^{I}p_i\sum_{j=1}^{I}\frac{(\pi^j_{n,x}(p_m))_i}{p_i}\psi^j_{n,\;x}(p_m)f(\pi^j_{n,x}(p_m))
			\\
			& =\sum_{j=1}^{I}\sum_{i=1}^{I}p_i\frac{(\pi^j_{n,x}(p_m))_i}{p_i}\psi^j_{n,\;x}(p_m)f(\pi^j_{n,x}(p_m))
			\\
			\nonumber
			&=\sum_{j=1}^{I}f(\pi^j_{n,x}(p_m))\psi^j_{n,\;x}(p_m)\Big(\sum_{i=1}^{I}(\pi^j_{n,x}(p_m))_i\Big)
			\\
			\nonumber
			&=\sum_{j=1}^{I}f(\pi^j_{n,x}(p_m))\psi^j_{n,\;x}(p_m),
			\end{align}
			since $\sum_{i=1}^{I}(\pi_{n,x}(p_m))_i = 1$. On noting \eqref{eq:martingale0}
			the statement follows directly from (\ref{frep})
			and the fact that $\mathcal{N}^h_{n,x}\subseteq\mathcal{N}^h$.
		\end{proof}
	}
	
	
	\begin{prop}\label{prop:sup}
		Every accumulation point $w$ of the sequence $\{ V_\hht^\hp\}_{\hht,\;\hp>0}$
		is a viscosity supersolution of \eqref{eq:1} on $[0,\;T]\times\RR^d\times\overline{\Delta(I)}$.
	\end{prop}
	\begin{proof}
		Let $\phi:[0,\;T]\times\RR^d\times\overline{\Delta(I)}\rightarrow\RR$ be a smooth test function, such that $w-\phi$ has a strict global minimum at $(\ot,\;\ox,\;\op)$ with $(w-\phi)(\ot,\;\ox,\;\op) = 0$. We  show that $\phi$ satisfies \eqref{eq:14} at $(\ot,\;\ox,\;\op)$.
		
		There exists a sequence $(\ott,\;\oxt,\;\opt)_\hht$ such that $\ott = l_\tau\hht \in\Pi_\hht$, $l_\tau\in\mathbb{N}$ 
		converges to $\ot$, $\opt\in\mN^{h}$  converges to $\op$, and $\oxt$ to $ \ox$ for $\tau, h\rightarrow 0$; also, $ V^\hp_\hht -\phi$ has a global minimum at $(\ott,\;\oxt,\;\opt)$
		{on $\Pi_\hht \times \mathbb{R}^d\times \mN^{h}$}.
		
		Define $\phi_\hht^\hp := \phi  +(V_\hht^\hp-\phi)(\ott,\;\oxt,\;\opt)$. Then for all $x\in\RR^d$ and $p_m\in\mN^\hp$ we have
		\begin{equation}\label{eq:32}
		\begin{split}
		(V_\hht^\hp-\phi_\hht^\hp)(\ott + \hht,\;x,\;p_m)\geq (V_\hht^\hp - \phi_\hht^\hp)(\ott,\;\oxt,\;\opt) =0.
		\end{split}
		\end{equation}
		
		Set 
		\begin{equation*}
		\oX_{n + 1} := \oxt + \sigma(\ott,\;\oxt)\xi_{l_\tau}\sqrt{\hht};
		\end{equation*}
		and for each $m\in\big\{1,\ldots,M\big\}$:
		\begin{align*}
		\oZ^m_{n}(\oxt) 
		&:=\frac{1}{{\hht}}\EE\big[V_{\hht}^{\hp}(\ott + \hht,\;\oX_{n + 1}, \;p_m)(\sigma(\ott,\;\oxt))^{-T}\xi_{l_\tau}\sqrt{\hht}\big],
		\\
		\oY^m_{n}(\oxt) &:=\EE\big[V^\hp_{\hht}(\ott + \hht,\;\oX_{n+1},\;p_m)\big] + \hht H\big(\ott,\;\oxt,\;\oZ^m_{n}(\oxt),\;p_m\big).
		\end{align*}
		From the non convex data set $\mY_{n}(\oxt) := \big\{\oY^1_{n}(\oxt), \ldots, \oY^M_{n}(\oxt)\big\}$, we define
		\begin{equation*}
		V^{\hp}_{\hht}(\ott,\;\oxt,\;\opt) :=\vexp\big[\mY_{n}(\oxt)\big](\opt).
		\end{equation*}
		
		We can assume that $\lbdmin\left(\op,\;\frac{\partial^2\phi_\hht^\hp}{\partial p^2}(\ot,\;\ox,\op)\right)>0$, otherwise \eqref{eq:14} is always true. Thus, there exist $\delta,\;\eta>0$ such that for all $\hht$ small enough we have
		\begin{equation}\label{eq:31}
		\begin{split}
		&\Big\la \frac{\partial^2\phi_\hht^\hp}{\partial p^2}(t,\;x,\;p)z,\;z\Big\ra > 4\delta\lvert z\rvert^2
		\\
		&\hspace{50pt}\forall\,  (x,\;p)\in B_\eta(\oxt,\;\opt), \forall\, t\in[\ott,\;\ott+\hht],\;\forall\,z\in T_{\overline{\Delta(I)}(\opt)}.
		\end{split}
		\end{equation}
		Furthermore, we assume without loss of generality that outside of  $B_\eta(\ott,\;\oxt,\;\opt)$, $\phi_\hht^\hp$ is still convex on $\Delta(I)$. Thus for any $p_m\in\mN^{h}$ it holds
		\begin{align}\notag
		V_\hht^\hp(\ott + \hht,\;x,\;p_m)&\geq \phi_\hht^\hp(\ott + \hht,\;x,\;p_m)
		\\
		\label{eq:33}
		&\geq \phi_\hht^\hp(\ott + \hht,\;x,\;\opt) + \Big\la  \frac{\partial \phi_\hht^\hp}{\partial p}(\ott + \hht,\;x,\;\opt),\; p_m-\opt\Big\ra .
		\end{align}
		
		The rest of the proof consists of 4 steps.
		
		\medskip
		
		\begin{steps}
			\item We prove that for any $p\in\overline{\Delta(I)}$, the following inequality holds
			\begin{equation}\label{eq:34}
			\begin{split}
			\EE&\big[V_\hht^\hp(\ott + \hht,\;x,\;p)\big]
			\\
			&\hspace{10pt}\geq \EE\big[ \phi_\hht^\hp(\ott + \hht,\;x,\;\opt)\big] + \Big\la \frac{\partial \phi_\hht^\hp}{\partial p} (\ott + \hht,\;x,\;\opt),\; p-\opt\Big\ra+ \delta\lvert p -\opt\rvert^2.
			\end{split}
			\end{equation}
			
			\medskip
			
			Fix $(x,\;p)\in B_\eta(\oxt,\;\opt)$ and $t\in[\ott,\;\ott+\hht]$. Since $\phi_\hht^\hp$ is  smooth, we expand $\phi_\hht^\hp$ into a Taylor--Lagrange expansion up to the order 2 and obtain for some $a\in B_\eta(\opt)$
			\begin{equation*}
			\phi_\hht^\hp(t,\;x,\;p) = \phi_\hht^\hp(t,\;x,\;\opt) + \Big\la\frac{\partial \phi_\hht^\hp}{\partial p}(t,\;x,\;\opt),\;p-\opt\Big\ra+\frac{1}{2}\Big\la\frac{\partial^2 \phi_\hht^\hp}{\partial p^2}(t,\;x,\;a)(p-\opt),\;p-\opt\Big\ra,
			\end{equation*}
			which thanks to \eqref{eq:31} gives
			\begin{equation}
			\label{eq:45}
			\phi_\hht^\hp(t,\;x,\;p)\geq \phi_\hht^\hp(t,\;x,\;\opt) + \Big\la\frac{\partial \phi_\hht^\hp}{\partial p}(t,\;x,\;\opt),\;p-\opt\Big\ra+2 \delta\llvert p-\opt\rrvert^2.
			\end{equation}
			
			For any $p\in\overline{\Delta(I)}\setminus\Int\left(B_\eta(\opt)\right)$. We set $\tp:=\opt +\eta(p-\opt)/\llvert p-\opt\rrvert$. Since the function $V_\hht^\hp$ is convex in the variable $p$, then the subgradient of $V_\hht^\hp(\ott,\;\oxt,\;\cdot)$ at $p$, denoted by $\partial^{-}V_\hht^\hp(\ott,\;\oxt,\;p)$ is not an empty set. Let $\hat{p}\in\partial^{-}V_\hht^\hp(\ott,\;\oxt,\;\hat{p})$, we have by definition of the subgradient
			\begin{align}\label{eq:46}
			V_\hht^\hp(\ott,\;\oxt,\;p)\geq V_\hht^\hp(\ott,\;\oxt,\;\tp) + \left\la \hat{p},\;p-\tp\right\ra.
			\end{align}
			By \eqref{eq:32} we have $(V_\hht^\hp-\phi_\hht^\hp)(\ott,\;\oxt,\;\tp)\geq 0$. Since $\tp \in B_\eta(\opt)$, using the inequalities \eqref{eq:45} and \eqref{eq:46} we obtain
			\begin{align}
			\notag
			\begin{split}
			V_\hht^\hp(\ott,\;\oxt,\;p)&\geq \phi_\hht^\hp(\ott,\;\oxt,\;\opt)+\Big\la\frac{\partial \phi_\hht^\hp}{\partial p}(\ott,\;\oxt,\;\opt),\;\tp-\opt\Big\ra
			+2 \delta\llvert \tp-\opt\rrvert^2+\big\la \hat{p},\;p-\tp\big\ra
			\end{split}
			\\
			\label{eq:sup4}
			\begin{split}
			&\geq \phi_\hht^\hp(\ott,\;\oxt,\;\opt)+\Big\la\frac{\partial \phi_\hht^\hp}{\partial p}(\ott,\;\oxt,\;\opt),\;p-\opt\Big\ra
			\\
			&\hspace{120pt}+2 \delta\llvert \tp-\opt\rrvert^2+\Big\la \hat{p}-\frac{\partial \phi_\hht^\hp}{\partial p}(\ott,\;\oxt,\;\opt),\;p-\tp\Big\ra.
			\end{split}
			\end{align}
			We show that the last term in the right hand side of \eqref{eq:sup4} is positive. By taking $p=\tp$ in \eqref{eq:33} and taking $p=\opt$ in \eqref{eq:46}  we have 
			\begin{align}
			\label{eq:sup1}
			V_\hht^\hp(\ott,\;\oxt,\;\tp)&\geq V_\hht^\hp(\ott,\;\oxt,\;\opt) + \Big\la  \frac{\partial \phi_\hht^\hp}{\partial p}(\ott,\;\oxt,\;\opt), \;\tp-\opt\Big\ra,
			\\
			\label{eq:sup2}
			V_\hht^\hp(\ott,\;\oxt,\;\opt)&\geq V_\hht^\hp(\ott,\;\oxt,\;\tp) + \big\la \hat{p},\;\opt-\tp\big\ra.
			\end{align}
			We sum up \eqref{eq:sup1} and \eqref{eq:sup2}. Then with the choice of $\tp$ we made, it follows that
			\begin{align*}
			0&\geq \Big\la \hat{p} -\frac{\partial \phi_\hht^\hp}{\partial p}(\ott,\;\oxt,\;\opt),\;\opt-\tp\Big\ra = \frac{\eta}{\llvert p-\opt \rrvert}\Big\la \hat{p} -\frac{\partial \phi_\hht^\hp}{\partial p}(\ott,\;\oxt,\;\opt),\;\opt-p\Big\ra.
			\end{align*}
			Thanks to the choice of $\tilde{p}$ we note that $\left(\frac{ \llvert p-\opt\rrvert}{\llvert p-\opt\rrvert+\eta}\right)(\tp-p)= (\opt-p)$. It implies that
			\begin{equation}\label{eq:sup3}
			\Big\la \hat{p} -\frac{\partial \phi_\hht^\hp}{\partial p}(\ott,\;\oxt,\;\opt),\;p-\tp\Big\ra\geq  0.
			\end{equation}
			Thus, substituting \eqref{eq:sup3} into \eqref{eq:sup4} gives
			\begin{align*}
			V_\hht^\hp(\ott,\;\oxt,\;p)&\geq V_\hht^\hp(\ott,\;\oxt,\;\opt)+\Big\la\frac{\partial \phi_\hht^\hp}{\partial p}(\ott,\oxt
			,\opt),\;p-\opt\Big\ra+2 \delta\llvert \tp-\opt\rrvert^2
			\\
			&=V_\hht^\hp(\ott,\;\oxt,\;\opt)+\Big\la\frac{\partial \phi_\hht^\hp}{\partial p}(\ott,\oxt
			,\opt),\;p-\opt\Big\ra+2 \delta\eta^2.
			\end{align*}
			After taking the limit $\hht,\hp\to 0$, we obtain for all $p\in\Delta(I)\setminus\Int\left(B_\eta(\op)\right)$ that
			\begin{align}\label{eq:49}
			w(\ot,\;\ox,\;p)\geq w(\ot,\;\ox,\;\op)+\left\la\frac{\partial \phi}{\partial p}(\ot,\;\ox,\;\op),\;p-\op\right\ra+2 \delta\eta^2.
			\end{align}
			Next, suppose that \eqref{eq:34} does not hold for a $p\in\overline{\Delta(I)}$. Thus there exists a sequence $(\hht,\;x_{\hht},\;p_{\hht})_{\hht}$ with $p_{\hht}\in\overline{\Delta(I)}\setminus B_\eta(\opt)$ such that $(\hht,\;x_{\hht},\;p_{\hht})\to (0,\;0,\;p)$ for $\hht,\,\hp\to 0$ and
			\begin{equation}
			\label{eq:sup5}
			\begin{split}
			&V_\hht^\hp(\ot_{\hht} + \hht, \;\ox_{\hht}
			+ x_{\hht}, \;p_{\hht})
			\\
			&\hspace{10pt}< \phi_\hht^\hp(\ot_{\hht} + \hht,\;\ox_{\hht}
			+ x_{\hht},\;\op_{\hht}) 
			+ \Big\la \frac{\partial\phi_\hht^\hp}{\partial p} (\ott + \hht,\;\oxt + x_{\hht},\;\opt),\; p_{\hht}-\opt\Big\ra
			+ \delta\lvert p_{\hht} -\opt\rvert^2.
			\end{split}
			\end{equation}
			For $\hht,\,\hp\to 0$, $p\in\overline{\Delta(I)}\setminus B_\eta(\op)$ it follows from \eqref{eq:sup5} that
			\begin{equation*}
			\begin{split}
			&w(\ot,\;\ox 
			,\;p)< \phi(\ot,\;\ox
			,\;\op) + \Big\la \frac{\partial\phi}{\partial p} (\ot,\;\ox,\;\op), \;p-\op\Big\ra + \delta \lvert p-\op \rvert^2.
			\end{split}
			\end{equation*}
			which contradicts the inequality \eqref{eq:49}. Hence, \eqref{eq:34} holds.
			
			\item We prove that  for any $p_m\in\mN^\hp$ we have
			\begin{equation}\label{eq:sup14}
			\begin{split}
			&\EE\big[ V^\hp_\hht(\ott + \hht,\;\oX_{n+1},\;p_m)\big]
			\\
			&\hspace{25pt}\geq \EE\bigg[   \phi_\hht^\hp(\ott + \hht,\;\oX_{n+1},\;\opt) + \Big\la \frac{\partial \phi_\hht^\hp}{\partial p} (\ott + \hht,\;\oX_{n+1},\;\opt),\; p_m-\opt\Big\ra
			\\
			&\hspace{200pt}+ \delta\lvert p_m -\opt\rvert^2 \ind_{\lvert \oX_{n + 1} - \oxt\rvert<\eta}\bigg].
			\end{split}
			\end{equation}
			
			\medskip
			
			With the estimates \eqref{eq:34} and \eqref{eq:32} we have for $\hht$ small enough and for all $p_m\in\mN^\hp$
			{
				\begin{align}
				\notag
				\begin{split}
				\EE&\big[ V^\hp_\hht(\ott + \hht,\;\oX_{n+1},\;p_m)\big]
				\\[5pt]
				&= \EE\bigg[ V^\hp_\hht(\ott + \hht,\;\oX_{n+1},\;p_m)\ind_{\lvert \oX_{n + 1} - \oxn\rvert<\eta}\bigg] 
				\\
				&\hspace{20pt}+ \EE\bigg[ V^\hp_\hht(\ott + \hht,\;\oX_{n+1},\;p_m)\ind_{\lvert \oX_{n + 1} - \oxn\rvert\geq\eta}\bigg]
				\end{split} 
				\\
				\notag
				\begin{split}
				&\geq \EE\bigg[\Big(  \phi_\hht^\hp(\ott + \hht,\;\oX_{n + 1},\;\opt)  + \Big\la \frac{\partial \phi_\hht^\hp}{\partial p} (\ott + \hht,\;\oX_{n+1},\;\opt),\; p_m-\opt\Big\ra
				\\
				&\hspace{10pt}+ \delta\lvert p_m -\opt\rvert^2 \Big)\ind_{\lvert \oX_{n + 1} - \oxn\rvert<\eta}\bigg]
				+ \EE\bigg[ V^\hp_\hht(\ott + \hht,\;\oX_{n+1},\;p_m)\ind_{\lvert \oX_{n + 1} - \oxn\rvert\geq\eta}\bigg]
				\end{split}
				\\
				\notag
				\begin{split}
				&\geq \EE\bigg[\Big(  \phi_\hht^\hp(\ott + \hht,\;\oX_{n+1},\;\opt) + \Big\la \frac{\partial \phi_\hht^\hp}{\partial p} (\ott + \hht,\;\oX_{n+1},\;\opt),\; p_m-\opt\Big\ra
				\\
				&\hspace{10pt}+ \delta\lvert p_m -\opt\rvert^2 \Big)\ind_{\lvert \oX_{n + 1} - \oxn\rvert<\eta}\bigg]
				+ \EE\bigg[ \phi^\hp_\hht(\ott + \hht,\;\oX_{n+1},\;p_m)\ind_{\lvert \oX_{n + 1} - \oxn\rvert\geq\eta}\bigg]
				\end{split}
				\\
				\label{eq:sup24}
				\begin{split}
				&= \EE\bigg[\phi_\hht^\hp(\ott + \hht,\;\oX_{n+1},\;\opt) + \Big\la \frac{\partial \phi_\hht^\hp}{\partial p} (\ott + \hht,\;\oX_{n+1},\;\opt),\; p_m-\opt\Big\ra
				\\
				&\hspace{5pt}+ \delta\lvert p_m -\opt\rvert^2 \ind_{\lvert \oX_{n + 1} - \oxn\rvert<\eta}\bigg]
				\\
				&\hspace{5pt}+ \EE\bigg[ \Big(\phi^\hp_\hht(\ott + \hht,\;\oX_{n+1},\;p_m)-\phi_\hht^\hp(\ott + \hht,\;\oX_{n+1},\;\opt)
				\\
				&\hspace{5pt}-\bigg\la \frac{\partial \phi_\hht^\hp}{\partial p} (\ott + \hht,\;\oX_{n+1},\;\opt),\; p_m-\opt\Big\ra\Big)\ind_{\lvert \oX_{n + 1} - \oxn\rvert\geq\eta}\bigg].
				\end{split}
				\end{align}	
			}
			We recall that $\phi^\hp_\hht$ is convex in the variable $p$, which implies that
			\begin{equation}\label{eq:sup23}
			\begin{split}
			&\phi^\hp_\hht(\ott + \hht,\;\oX_{n+1},\;p_m)-\phi_\hht^\hp(\ott + \hht,\;\oX_{n+1},\;\opt)
			\\
			&\hspace{75pt}-\bigg\la \frac{\partial \phi_\hht^\hp}{\partial p} (\ott + \hht,\;\oX_{n+1},\;\opt),\; p_m-\opt\Big\ra\geq 0.
			\end{split}
			\end{equation} 
			Hence, from \eqref{eq:sup23} and \eqref{eq:sup24} the assertion \eqref{eq:sup14} holds for all $p_m\in\mN^\hp$.
			
			\item\label{item:sup3}Next we establish an estimate for $\fp_{n+1}:=\fp_{l_\tau+1}^{\oxt,\;\opt}$ where $\fp_{l_\tau+1}^{\oxt,\;\opt}$ is defined as a one step martingale as in \cref{defn:onestepfb} with initial data $(\ott,\;\oxt,\;\opt)$.
			
			\medskip
			
			Note that by \cref{lem:dpp} it holds
			\begin{equation}\label{eq:36}
			\begin{split}
			&V_\hht^\hp(\ott,\;\oxt,\;\opt) = \EE \Big[V_\hht^\hp(\ott + \hht,\;\oX_{n+1},\;\fp_{n+1}) 
			\\
			&\hspace{20pt}+ \hht H\Big(\ott,\;\oxt,\;\frac{1}{\hht}\EE\Big[ V_\hht^\hp(\ott + \hht,\;\oX_{n+1},\;\fp_{n+1})(\sigma(\ott,\;\oxt))^{-T}\xi_{n}\sqrt{\hht}\Big],\;\fp_{n+1}\Big)\Big].
			\end{split}
			\end{equation}
			
			We replace the first term of the right hand side of \eqref{eq:36} with \eqref{eq:sup14} (for $p=\fp_{n+1}$) and obtain using \eqref{eq:32} for small enough $\hht,\, \hp>0$ that
			\begin{align}
			\notag
			\begin{split}
			0&\geq  \EE\Big[ \phi_\hht^\hp(\ott + \hht,\;\oX_{n+1},\;\opt)-\phi_\hht^\hp(\ott,\;\oxt,\;\opt)\Big]
			\\
			&\qquad+\hht\EE\Big[H\Big(\ott,\;\oxt,\;\frac{1}{\hht}\EE\Big[ V_\hht^\hp(\ott + \hht,\;\oX_{n+1},\;\fp_{n+1})(\sigma(\ott,\;\oxt))^{-T}\xi_{n}\sqrt{\hht}\Big],\;\fp_{n+1}\Big)\Big]
			\\
			&\qquad+\EE \Big[\Big\la\frac{\partial\phi_\hht^\hp}{\partial p} (\ott + \hht,\;\oX_{n+1},\;\opt), \;\fp_{n+1}-\opt\Big\ra\Big] + \delta\EE\Big[\ind_{\lvert \oX_{n+1} - \oxt\rvert<\eta}\lvert \opt -\fp_{n+1}\rvert^2\Big]
			\end{split}
			\\
			\label{eq:38}
			&\qquad=: \I + \II + \III + \IV.
			\end{align}
			We estimate the right-hand side of \eqref{eq:38}.
			
			We note that since $\phi$ is smooth, from the Taylor expansion it follows
			\begin{equation}\label{eq:sup6}
			\begin{split}
			&\EE \big[\phi(\ott +\hht,\;\oX_{n+1},\;\opt)-\phi(\ott,\;\oxt,\;\opt)\big] 
			\\
			&\hspace{25pt}=  \big[\dt\phi(\ott,\;\oxt,\;\opt)
			+\frac12\Tr\big(\sigma\sigma^T(\ott,\;\oxt)D^2_x \phi(\ott,\;\oxt,\;\opt)\big)\big]\hht + \hht \mO(\hht)
			\\
			&\hspace{25pt}\leq C\hht + \hht \mO(\hht).
			\end{split}
			\end{equation}
			Hence, using \eqref{eq:sup6} we obtain
			\begin{equation}
			\label{eq:sup7}
			\begin{split}
			\I &\leq C\hht +\hht \mO(\hht).
			\end{split}
			\end{equation}
			
			Using the same arguments leading to (\ref{eq:xlip26}), (\ref{eq:xlip25}) we estimate the second term as
			\begin{equation}\label{eq:sup8}
			\II \leq C\hht  \Big(1 + \frac{1}{\hht}\EE\Big[ V_\hht^\hp(\ott + \hht,\;\oX_{n+1},\;\fp_{n+1})(\sigma(\ott,\;\oxt))^{-T}\xi_{n}\sqrt{\hht}\Big]\Big)\leq C\hht.
			\end{equation}
			
			Due to the independence of $\fp_{n+1}$ and $\xi_{l_\tau}$ and by the martingale property $\EE\big[\fp_{n+1}\big] = \opt$, we have
			\begin{equation}
			\label{eq:sup9}
			\III\leq \Big\la\EE  \Big[\frac{\partial V_\hht^\hp}{\partial p} (\ott + \hht,\;\oX_{n+1},\;\opt)\Big],\; \EE\Big[\fp_{n+1}-\opt\Big] \Big\ra= 0.
			\end{equation}
			
			By the Markov's inequality it follows
			\begin{align}
			\notag
			\IV&=\delta \PP \Big[{\lvert \sigma(\ott,\;\oxt)\xi_{l_\tau}\sqrt{\hht} \vert<\eta}\Big]\EE\Big[\lvert \opt -\fp_{n+1}\rvert^2\Big]
			\\
			\notag
			&\geq  \delta\Big(1-\frac{1}{\eta}\EE\Big[\lvert \sigma(\ott,\;\oxt)\xi_{l_\tau}\sqrt{\hht}\vert\Big] \Big)\EE\Big[\lvert \opt -\fp_{n+1}\rvert^2\Big]
			\\
			\label{eq:sup10}
			&\geq  C(\delta,\;\eta)\left(1-\hht^{1/2}\right)\EE\Big[\lvert \opt -\fp_{n+1}\rvert^2\Big].
			\end{align}
			
			We substitute the estimates \eqref{eq:sup7}, \eqref{eq:sup8}, \eqref{eq:sup9}, and \eqref{eq:sup10} obtained for $\I$, $\II$, $\III$, and $\IV$ into \eqref{eq:38}. And get
			\begin{equation*} 
			\EE\Big[\lvert \opt -\fp_{n+1}\rvert^2\Big]\leq C(\delta,\;\eta)\left(\frac{ \hht }{1-\hht^{1/2}}\right).
			\end{equation*}
			For $0<\hht<1/2$ small enough we have 
			\begin{equation}\label{eq:39}
			\EE\Big[\lvert \opt -\fp_{n+1}\rvert^2\Big]\leq C(\delta,\;\eta)\hht.
			\end{equation}
			
			\item In last step we show that
			\begin{equation*}
			0\geq \dt\phi(\ot,\;\ox,\;\op)
			+\frac12\Tr\big(\sigma\sigma^T(\ot,\;\ox)D^2_x \phi(\ot,\;\ox,\;\op)\big) + H\big(\ot,\;\ox,\;D_x\phi(\ot,\;\ox,\;\op),\;\op\big).
			\end{equation*}	
			
			\medskip
			
			By \cref{lem:dpp} it holds
			\begin{equation}\label{eq:37}
			\begin{split}
			&V_\hht^\hp(\ott,\;\oxt,\;\opt) = \EE \bigg[V_\hht^\hp(\ott + \hht,\;\oX_{n+1},\;\fp_{n+1}) 
			\\
			&\hspace{10pt}+ \hht H\Big(\ott,\;\oxt,\;\frac{1}{\hht}\EE\Big[ V_\hht^\hp(\ott + \hht,\;\oX_{n+1},\;\fp_{n+1})(\sigma(\ott,\;\oxt))^{-T}\xi_{l_\tau}\sqrt{\hht}\Big],\;\fp_{n+1}\Big)\bigg].
			\end{split}
			\end{equation}
			{
				We recall {\eqref{eq:32}} and apply \cref{lem:monotonicity} for $\phi_1(\cdot)=V^\hp_\hht(\ott + \hht,\;\cdot,\;p_m)$ and $\phi_2(\cdot)=\phi_\hht^\hp(\ott + \hht,\;\cdot,\;p_m)$  to get
				\begin{equation}\label{eq:25}
				\begin{split}
				&\EE \Big[V_\hht^\hp(\ott + \hht,\;\oX_{n+1},\;p_m) \Big]
				\\
				&\hspace{5pt}+ \hht H\Big(\ott,\;\oxt,\;\frac{1}{\hht}\EE\Big[ V_\hht^\hp(\ott + \hht,\;\oX_{n+1},\;p_m)(\sigma(\ott,\;\oxt))^{-T}\xi_{l_\tau}\sqrt{\hht}\Big],\;p_m\Big)
				\\
				&\geq  \EE\Big[ \phi_\hht^\hp(\ott + \hht,\;\oX_{n+1},\;p_m)\Big]
				\\
				&\hspace{5pt}+\hht H\Big(\ott,\;\oxt,\;\frac{1}{\hht}\EE\Big[ \phi_\hht^\hp(\ott + \hht,\;\oX_{n+1},\;p_m)(\sigma(\ott,\;\oxt))^{-T}\xi_{l_\tau}\sqrt{\hht}\Big],\;p_m\Big)-C\hht\sqrt{\hht}\,,
				\end{split}
				\end{equation}
				for all $p_m\in\mathcal{N}^h$.
				On noting that $\fp_{n+1}$ is a $\mathcal{N}^h$-value random variable which is independent from $\oX_{n+1}$, $\xi_{l_\tau}$ we deduce from \eqref{eq:25} that
				\begin{equation}\label{eq:sup15}
				\begin{split}
				&\EE \bigg[V_\hht^\hp(\ott + \hht,\;\oX_{n+1},\;\fp_{n+1}) 
				\\
				&+ \hht H\Big(\ott,\;\oxt,\;\frac{1}{\hht}\EE\Big[ V_\hht^\hp(\ott + \hht,\;\oX_{n+1},\;\fp_{n+1})(\sigma(\ott,\;\oxt))^{-T}\xi_{l_\tau}\sqrt{\hht}\Big],\;\fp_{n+1}\Big)\bigg]
				\\
				&\geq  \EE\bigg[ \phi_\hht^\hp(\ott + \hht,\;\oX_{n+1},\;\fp_{n+1})
				\\
				&+\hht H\Big(\ott,\;\oxt,\;\frac{1}{\hht}\EE\Big[ \phi_\hht^\hp(\ott + \hht,\;\oX_{n+1},\;\fp_{n+1})(\sigma(\ott,\;\oxt))^{-T}\xi_{l_\tau}\sqrt{\hht}\Big],\;\fp_{n+1}\Big)\bigg]
				\\
				&-C\hht\sqrt{\hht}.
				\end{split}
				\end{equation}
				Then, from \eqref{eq:37} and \eqref{eq:sup15}, it follows
				\begin{equation}\label{eq:sup19}
				\begin{split}
				&V_\hht^\hp(\ott,\;\oxt,\;\opt)\geq  \EE\bigg[ \phi_\hht^\hp(\ott + \hht,\;\oX_{n+1},\;\fp_{n+1})
				\\
				&\hspace{5pt}+\hht H\Big(\ott,\;\oxt,\;\frac{1}{\hht}\EE\Big[ \phi_\hht^\hp(\ott + \hht,\;\oX_{n+1},\;\fp_{n+1})(\sigma(\ott,\;\oxt))^{-T}\xi_{l_\tau}\sqrt{\hht}\Big],\;\fp_{n+1}\Big)\bigg]
				\\
				&\hspace{5pt}-C\hht\sqrt{\hht}.
				\end{split}
				\end{equation}
				By using in \eqref{eq:sup19} the last equality on the right hand side of \eqref{eq:32}, as well as the definition of $\phi_\hht^\hp$ (see above \eqref{eq:32}), we obtain
				\begin{equation}\label{eq:sup21}
				\begin{split}
				\phi(\ott,&\;\oxt,\;\opt)\geq  \EE\bigg[ \phi(\ott + \hht,\;\oX_{n+1},\;\fp_{n+1})
				\\
				&+\hht H\Big(\ott,\;\oxt,\;\frac{1}{\hht}\EE\Big[ \phi_\hht^\hp(\ott + \hht,\;\oX_{n+1},\;\fp_{n+1})(\sigma(\ott,\;\oxt))^{-T}\xi_{l_\tau}\sqrt{\hht}\Big],\;\fp_{n+1}\Big)\bigg]
				\\
				&-C\hht\sqrt{\hht}.
				\end{split}
				\end{equation}
			}
			
			The stochastic process $\fp_{n+1}$ is independent of $\oX_{n + 1}$ by construction. Since $\phi^\hp_\hht$ is convex and arbitrarily smooth, thanks to \eqref{eq:33} we obtain 
			\begin{align}
			\notag
			\EE\Big[\phi(\ott +& \hht,\;\oX_{n+1},\;\fp_{n+1})\Big]
			\\
			\notag
			&\geq\EE\Big[\phi(\ott + \hht,\;\oX_{n+1},\;\opt)\Big] + \Big\la\EE\Big[\frac{\partial \phi}{\partial p}(\ott + \hht,\;\oX_{n+1},\;\opt)\Big],\;\EE\Big[ \fp_{n+1}-\opt\Big]\Big\ra
			\\
			\label{eq:sup11}
			& = \EE\Big[\phi(\ott + \hht,\;\oX_{n+1},\;\opt)\Big].
			\end{align}
			
			Furthermore, by the Taylor expansion in $x$ and since the stochastic processes $\fp_{n+1}$ satisfies to the martingale property $\EE\big[\fp_{n+1}\big] = p_m$ and is independent of $\xi_{l_\tau}$, we obtain
			\begin{align}
			\notag
			&\frac{1}{\hht}\EE\Big[ \phi_\hht^\hp(\ott + \hht,\;\oX_{n+1},\;\fp_{n+1})(\sigma(\ott,\;\oxt))^{-T}\xi_{l_\tau}\sqrt{\hht}\Big]
			\\\notag
			\begin{split}
			&\hspace{7pt}= \frac{1}{\hht}\EE\Big[ \phi_\hht^\hp(\ott + \hht,\;\oxt,\;\fp_{n+1})(\sigma(\ott,\;\oxt))^{-T}\xi_{l_\tau}\sqrt{\hht}\Big] + \EE\Big[ D_x\phi_\hht^\hp(\ott + \hht,\;\oxt,\;\fp_{n+1})\Big] 
			\\
			&\hspace{140pt}+ \EE\big[(\sigma(\ott,\;\oxt))\xi_{l_\tau}(\sigma(\ott,\;\oxt))^{-T}\xi_{l_\tau}\big]\mO(\hht^{1/2})
			\end{split}
			\\
			\label{eq:sup12}
			& \hspace{7pt}= \EE\Big[ D_x\phi_\hht^\hp(\ott + \hht,\;\oxt,\;\fp_{n+1})\Big] + \mO(\hht^{1/2}).
			\end{align}
			We substitute \eqref{eq:sup11} and \eqref{eq:sup12} into \eqref{eq:sup21} to get
			\begin{equation}\label{eq:sup16}
			\begin{split}
			&0\geq \EE\bigg[\phi(\ott + \hht,\;\oX_{n+1},\;\opt) -\phi(\ott,\;\oxt,\;\opt)
			\\
			&\hspace{20pt}+\hht H\Big(\ott,\;\oxt,\;\EE\big[ D_x\phi_\hht^\hp(\ott + \hht,\;\oxt,\;\fp_{n+1})\big] + \mO(\hht^{1/2}),\;\fp_{n+1}\Big)\bigg]-C\hht\sqrt{\hht}.
			\end{split}
			\end{equation}
			
			Since $\phi$ is arbitrarily smooth, we can assume that $D_x\phi^\hp_\hht$ is Lipschitz continuous in $p$ which with \eqref{eq:39} imply that
			\begin{align}\label{eq:sup17}
			\notag
			&\EE\Big[\Big\lvert  D_x\phi_\hht^\hp(\ott + \hht,\;\oxt,\;\fp_{n+1}) -  D_x\phi_\hht^\hp(\ott + \hht,\;\oxt,\;\opt)\Big\rvert\Big]
			\\
			&\hspace{20pt}\leq C\EE\big[\lvert \fp_{n+1}-\opt\rvert\big]\leq C\Big(\EE\big[\lvert \fp_{n+1}-\opt\rvert^2\big]\Big)^{1/2}\leq C\hht^{1/2}.
			\end{align}	
			Combining \eqref{eq:4} and \eqref{eq:sup17} we have
			\begin{align}
			\notag
			&H\Big(\ott,\;\oxt,\; D_x\phi_\hht^\hp(\ott + \hht,\;\oxt,\;\opt),\;\opt\Big)
			\\\notag
			\begin{split}
			&\hspace{5pt}\leq 	\EE\Big[H\Big(\ott,\;\oxt,\;\EE\big[ D_x\phi_\hht^\hp(\ott + \hht,\;\oxt,\;\fp_{n+1})\big] + \mO(\hht^{1/2}),\;\fp_{n+1}\Big)\Big]
			\\
			&\hspace{20pt}+ \EE\Big[\Big\lvert  D_x\phi_\hht^\hp(\ott + \hht,\;\oxt,\;\fp_{n+1}) -  D_x\phi_\hht^\hp(\ott + \hht,\;\oxt,\;\opt)\Big\rvert\Big]  
			\\
			&\hspace{30pt}+ C\EE\big[\lvert \fp_{n+1}-\opt\rvert\big]+ \mO(\hht^{1/2})
			\end{split}
			\\\label{eq:sup13}
			&\leq  \EE\Big[H\Big(\ott,\;\oxt,\;\EE\big[D_x\phi(\ott + \hht,\;\oxt,\;\fp_{n+1})\big]  +\mO(\hht^{1/2}),\; \fp_{n+1}\Big)\Big]+ C\tau^{1/2} +\mO(\hht^{1/2}).
			\end{align}
			By the Taylor expansion we have
			\begin{equation}\label{eq:sup18}
			\begin{split}
			&\phi(\ott +\hht,\;\oX_{n+1},\;\opt) 
			\\
			&\hspace{5pt}=\phi(\ott,\;\oxt,\;\opt)+\Big[\dt\phi(\ott,\;\oxt,\;\opt)+\frac12\Tr\big(\sigma\sigma^T(\ott,\;\oxt)D^2_x \phi(\ott,\;\oxt,\;\opt)\big)\Big]\hht
			\\
			&\hphantom{=\phi(\ott,\;\oxt,\;\opt)+}+ \big[D_x\phi(\ott,\;\oxt,\;\opt)
			\big] \sigma(\ott,\;\oxt)\xi_{l_\tau}\sqrt{\hht} + \hht \mO(\hht) + \hht \mO(\hht^{1/2}).
			\end{split}
			\end{equation}
			Thanks to \eqref{eq:sup13} and \eqref{eq:sup18}, we can derive from \eqref{eq:sup16} that
			\begin{equation} \label{eq:sup20}
			\begin{split}
			C&\hht^{1/2}+ \mO(\hht^{1/2})+ \hht \mO(\hht^{1/2})\geq \dt\phi(\ott,\;\oxt,\;\opt)
			\\
			&+\frac12\Tr\big(\sigma\sigma^T(\ott,\;\oxt)D^2_x \phi(\ott,\;\oxt,\;\opt)\big) +   H\big(\ott,\;\oxt,\;D_x\phi(\ott + \hht,\;\oxt,\;\opt),\;\opt\big).
			\end{split}
			\end{equation}
			Since $(\ott,\;\oxt,\;\opt)\to (\ot,\;\ox,\;\op)$ for $\hht,\,\hp\to 0 $, it follows from \eqref{eq:sup20} that 
			\begin{equation*}
			0\geq \dt\phi(\ot,\;\ox,\;\op)
			+\frac12\Tr\big(\sigma\sigma^T(\ot,\;\ox)D^2_x \phi(\ot,\;\ox,\;\op)\big) + H\big(\ot,\;\ox,\;D_x\phi(\ot,\;\ox,\;\op),\;\op\big)
			\end{equation*}
			which concludes the proof.
		\end{steps}
	\end{proof}
	
	\section{Implementation and Computational studies}\label{sec_comput}
	In this section we present an implementable fully discrete version of Algorithm~\ref{algo:algo} where the discretization in the spatial variable 
	is realized via piecewise linear interpolation over a simplicial partition of the spatial domain.
	We also perform numerical simulations to demonstrate the properties of the proposed scheme.
	
	\subsection{Implementable full discretization}\label{sec_alg_impl}
	{For simplicity we describe the algorithm for the case of a bounded spatial domain $\mathcal{D} \subset \RR^d$.
		Let $\mT^{\hx}$ be a regular partition of $\mathcal{D}$ into open simplices $S$ with mesh size $\hx=\max_{S\in\mT^{\hx}} \{\diam (S)\}$ and
		denote the set of ``grid'' nodes of $\mT^{\hx}$ as $\mX:= \{x_1,\ldots,x_L\}$. The piecewise linear Lagrange basis associated with the partition
		$\mT^{\hx}$ is denoted as $\big\{\varphi_\ell\big\}_{\ell=1}^L$.
		
		Below we denote $\sigma_{n,\ell}:=\sigma(\tn,\;\xl)$, $\sigma^{-T}_{n,\ell}:=(\sigma^T(\tn,\;\xl))^{-1}$ and introduce
		the restriction of (\ref{xeuler1}) on the grid nodes $x_\ell\in \mX$ as
		\begin{equation*}
		\oX_{n + 1}^\ell:= \oX_{n + 1}^{\xl} = \xl +  \sigma_{n,\ell}\xi_{n}\sqrt{\hht} \qquad \ell=1,\dots,L.
		\end{equation*}
		The fully discrete algorithm computes the numerical approximation at the nodes $x_\ell \in \mT^{\hx}$, $\ell=1,\dots, L$. In general, the values $\oX_{n + 1}^\ell$
		do not coincide with the nodes $\mT^{\hx}$ and we obtain the intermediate value of the solution by
		linear interpolation over the simplicial mesh $\mT^{\hx}$.
		Given the fully discrete solution $\big\{V^{m,\ell}_{n}\big\}_{\ell=1}^L$ at $t_n$, $p_m$, $\{x_\ell\}_{\ell=1}^L$ its piecewise linear interpolant on $\mT^{\hx}$
		is expressed in terms of the piecewise linear Lagrange basis functions as
		\begin{equation}\label{interpol_x}
		V^{m,\hx}_{n}(x) = \sum_{i=1}^{L} V^{m,l}_{n}\varphi_l(x) \qquad x \in \mathbb{R}^d.
		\end{equation}
	}
	Hence, we obtain the following fully discrete version of Algorithm~\ref{algo:algo}.
	{
		\begin{algo}\label{algo:algo2}
			For $\xl\in\mX$, $\ell=1,\dots,L$  set $V_N^{m,\;\ell} = \la p_m,\;g(x_\ell)\ra$ for $p_m\in\mN^h$, $m= 1,\ldots,M$ and proceed for $n=N-1,\ldots,0$ as follows:
			\begin{enumerate}
				\item Forward step: for $x_\ell\in\mX$, $\ell = 1,\ldots,L$ compute:
				\begin{equation*}
				\oX_{n + 1}^\ell = \xl + \sigma_{n,\;\ell}\xi_{n}\sqrt{\hht};
				\end{equation*}
				\item Backward step: for $x_\ell\in\mX$, $\ell = 1,\ldots,L$ and $m = 1,\ldots,M$ set:
				\begin{align*}
				\oZ^{m,\;\ell}_{n} 
				&=\frac{1}{{\hht}}\EE\big[V^{m,\;\hx}_{n + 1}(\oX^{\ell}_{n + 1})\sigma^{-T}_{n,\;\ell}\xi_{n}\sqrt{\hht}\big],
				\\
				\oY^{m,\;\ell}_{n} &=\EE\big[V^{m,\;\hx}_{n + 1}(\oX^{\ell}_{n+1})\big] + \hht H\big(\tn,\;\xl,\;\oZ^{m,\;\ell}_{n},\;p_m\big);
				\end{align*}
				\item Convexification: for $\ell = 1,\ldots,L\}$ compute the discrete lower convex envelope $\big\{V^{1,\;\ell}_{n}, \ldots, V^{M ,\;\ell}_{n}\big\}$
				of $\{\oY^{1,\;\ell}_{n}, \ldots, \oY^{M,\;\ell}_{n}\}$ as:
				\begin{equation*}
				V^{m,\;\ell}_{n} =\vexp\big[\oY^{1,\;\ell}_{n}, \ldots, \oY^{M,\;\ell}_{n}\big](p_m),\quad m=1,\ldots,M.
				\end{equation*}
			\end{enumerate}
	\end{algo}}
	
	{
		There exist several efficient algorithms to compute the discrete lower convex envelope in step $(3)$ of the above algorithm. For $I=2$ one can directly solve the minimization problem \eqref{eq:vex0} for $m=1,\dots,M$,
		the corresponding algorithm is called Jarvis's march. For $I > 2$, where the direct minimization via \eqref{eq:vex0} becomes inefficient, one can employ more efficient convex hull algorithm such as
		the beneath-beyond or divide-and-conquer algorithms,
		or the Quickhull algorithm, cf. \cite{preparata1985computational} and \cite{barber1996the}. 
		
	}

	\subsection{Numerical experiments}
	In the numerical experiments below we take $I = 2$, $d = 1$, $T = 0.5$. 
	{We eliminate one probability variable from the solution by parametrizing $\Delta(2) = (p,\;1-p)$ for $p\in(0,\;1)$ and consider the transformed solution $V:=V(\cdot,\cdot,p)$ for $p\in(0,1)$.}
	We set $\sig(x) = \sigma_0 x(1-x)$, ($\sigma_0 >0$), and $H(x,\;p)=\sin(2\pi p)\cos(5\pi x) - \cos(5\pi p )\sin(3\pi x)$
	and consider a simplified version of (\ref{eq:1})
	\begin{equation}\label{eq:num1}
	\min\Big\{ \dt V + \tfrac{1}{2}\sig^2(x)\frac{\partial^2 V}{\partial x^2} + H(x,\;p),\;\lambda_{\min}\left(p, \frac{\partial^2 V}{\partial p^2}\right)\Big\} =0\, .
	\end{equation}
	Due to the choice of the diffusion $\sigma$ we may restrict the spatial domain to the interval $[0,1]$ which is partitioned uniformly into line segments, i.e., $\mathcal{T}^{\Delta x} = \{(x_{\ell-1}, x_{\ell})\}_{\ell=1}^{L}$, $x_\ell=\ell\Delta x$
	with the mesh size $\Delta x = 1/L$. Similarly, we partition the probability domain $[0,1]$ uniformly into segments with mesh size $h=1/M$ and nodes $p_m=mh$, $m=0,\dots,M$ and the time interval $[0,T]$
	with time-step size $\tau = 1/N$, $t_n=n\tau$, $n=1,\dots,N$.
	
	In the considered case, \cref{algo:algo2} has a particularly simple form, where we denote $\sigma_{\ell}:=\sigma(\xl)$ and express the expectations in step $(1)$ below explicitly
	since $\xi_n=\pm 1$. 
	\begin{algo}\label{algo:algo3}
		For $\ell=0,\dots, L$ set
		$V^{m,\;\ell}_N: = \la p_m,\;g(\xl)\ra$
		%
		{and $n  = N-1,\ldots,0$ proceed as follows}: 
		\begin{enumerate}
			\item For $\ell=0,\ldots,L$, $m=0,\ldots,M$:
			\begin{equation*}
			\oY^{m,\;\ell}_{n} =\frac{V^{m,\;\hx}_{n + 1}(\xl + \sigma_{\ell}\sqrt{\hht}) + V^{m,\;\hx}_{n + 1}(\xl - \sigma_{\ell}\sqrt{\hht})}{2} + \hht H(\xl,\;p_m);
			\end{equation*}
			\item For $\ell=0,\ldots,L$, $m = 1,\ldots,M-1$ compute:
			\begin{align*}
			\displaystyle V^{m,\; \ell }_n &=\min_{k\in\{1,\ldots,M-m\}}  \Big\{(1-\tfrac{1}{k})\oY^{m,\;\ell}_{n} + \tfrac{1}{k}\oY^{m + k,\;\ell}_{n}\Big\}.
			\end{align*}
		\end{enumerate}
	\end{algo}

	The numerical solution computed for $\sigma_0 = 0.5$, $N = 25$, $L = 100$, and $M = 100$ is displayed in \cref{fig:vexV}.
	
	\begin{figure}[hbtp!]
		\centering
		\subfloat[\label{fig:vexV}Obtacle on]{\includegraphics[width= 0.4\textwidth]{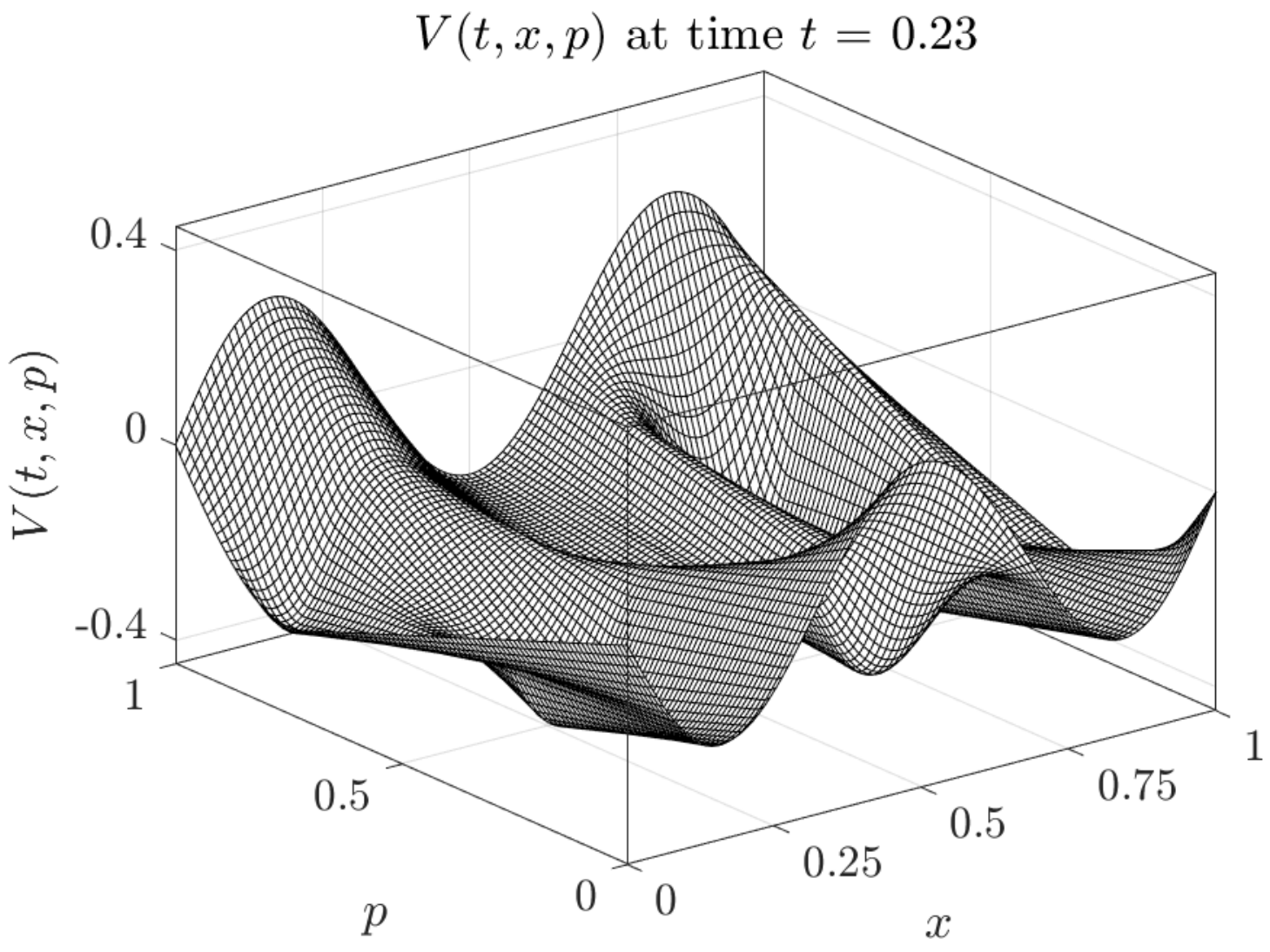}}
		\quad
		\subfloat[\label{fig:nonvexV}Obstacle off]{\includegraphics[width= 0.4\textwidth]{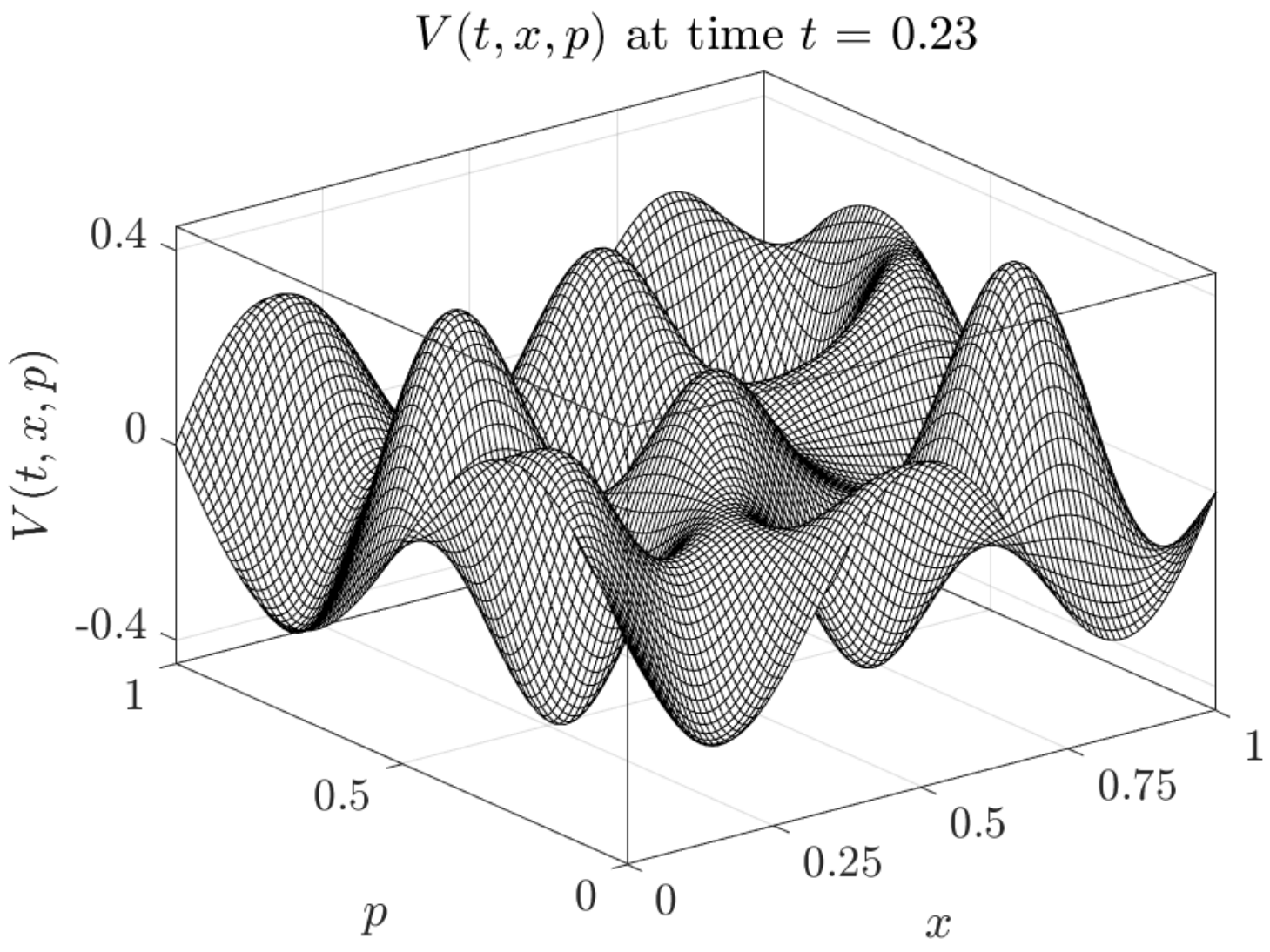}}
		\caption{(left) Numerical solution of \eqref{eq:num1} computed with \cref{algo:algo3}. (right) Numerical solution of \eqref{eq:num1} without the obstacle term $\lambda_{\min}\big(p, \frac{\partial^2 V}{\partial p^2}\big)$,
			computed with \cref{algo:algo3} without step $(2)$.}
	\end{figure}
	\begin{figure}[hbtp!]
		\centering
		\subfloat[\label{fig:Vprofile1}$(t,\;x) = (0.23,\;0.25)$]{\includegraphics[width= 0.3\textwidth]{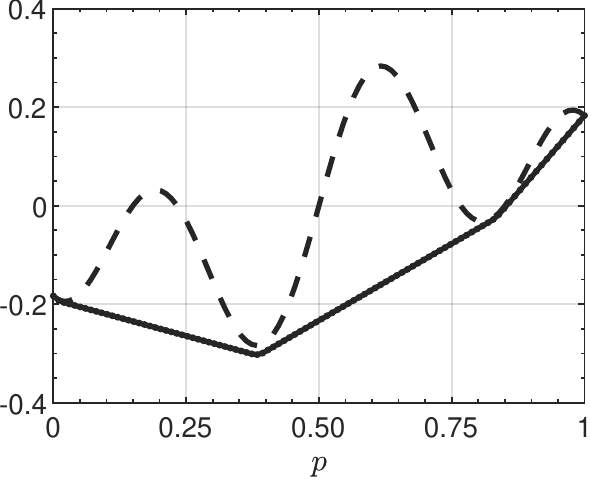}}
		\hspace{2pt}
		\subfloat[\label{fig:Vprofile2}$(t,\;x) = (0.23,\;0.50)$]{\includegraphics[width= 0.3\textwidth]{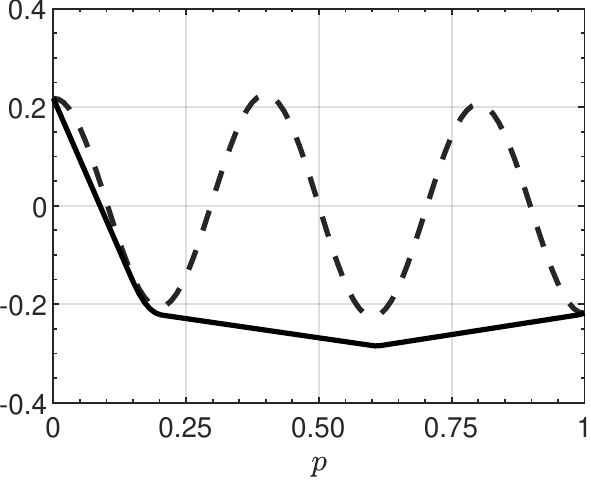}}
		\hspace{2pt}
		\subfloat[\label{fig:Vprofile3}$(t,\;x) = (0.23,\;0.75)$]{\includegraphics[width= 0.3\textwidth]{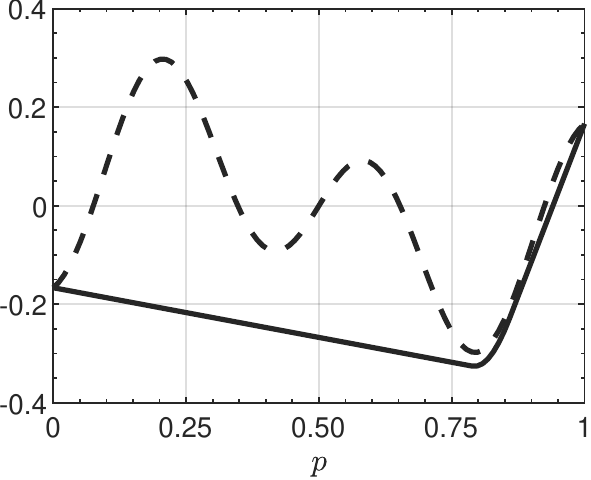}}
		\caption{Cross-section of the numerical solution in \cref{fig:vexV}   (solid line) and in \cref{fig:nonvexV} (dashed line) at $x = 0.5,\;0.50,\;0.75$.}
	\end{figure}
	
	{Since no analytic solution is known we determine the experimental order of convergence by using a reference solution
		$V_{\hht_{\reff}}$ which is computed for small discretization parameters $\hht_{\reff} = 1/384$ and $h = \hx_{\reff} = 1/1024$.}
	
	To study the error in the spatial discretization we fix $\hht = 1/50$, $\hp = 1/1024$
	and  vary $\hx = 1/L$ for $ L = 15, \;30,\; 60,\; 150,\; 300,\;600$. 
	The maximum error over all $x_\ell\in \mathcal{T}^\hx$ at $(t,\;p) = (0,\; 0.5)$ plotted against $\hx$ is displayed in \cref{fig:hx}.
	{We observe that the convergence in $\hx$ is roughly of first order.}
	
	Next, we study the error in the discretization in $p$.
	We fix $\hht = 1/50$, $\hx = 1/1024$ and $\hp=1/M$ for
	$ M = 15, \;30,\; 60,\; 150,\; 300,\;600$. 
	The maximum error over all $p_m\in \mathcal{N}_h$ at $(t,\;x) = (0,\; 0.5)$ plotted against $\hp$ is displayed in \cref{fig:hp}.
	Similarly as for the spatial discretization, we observe quadratic convergence in $\hp$.
	
	Finally, to study the error if the time-discretization
	we fix $\hx = \hp = 1/1024$ and vary $\hht = 1/N$ for $ N = 3,\; 6,\; 12,\; 24,\;48,\;86$. 
	The maximum error over all tile levels $t_n$, $n=1,\dots,N$ at $(x,\;p) = (0.5,\; 0.5)$ plotted against $\hp$ is displayed in \cref{fig:ht}.
	The convergence of the discretization with respect to $\hht$ is of linear order.
	
	\begin{figure}[hbtp!]
		\centering
		\subfloat[\label{fig:hx}]{\includegraphics[width= 0.29\textwidth]{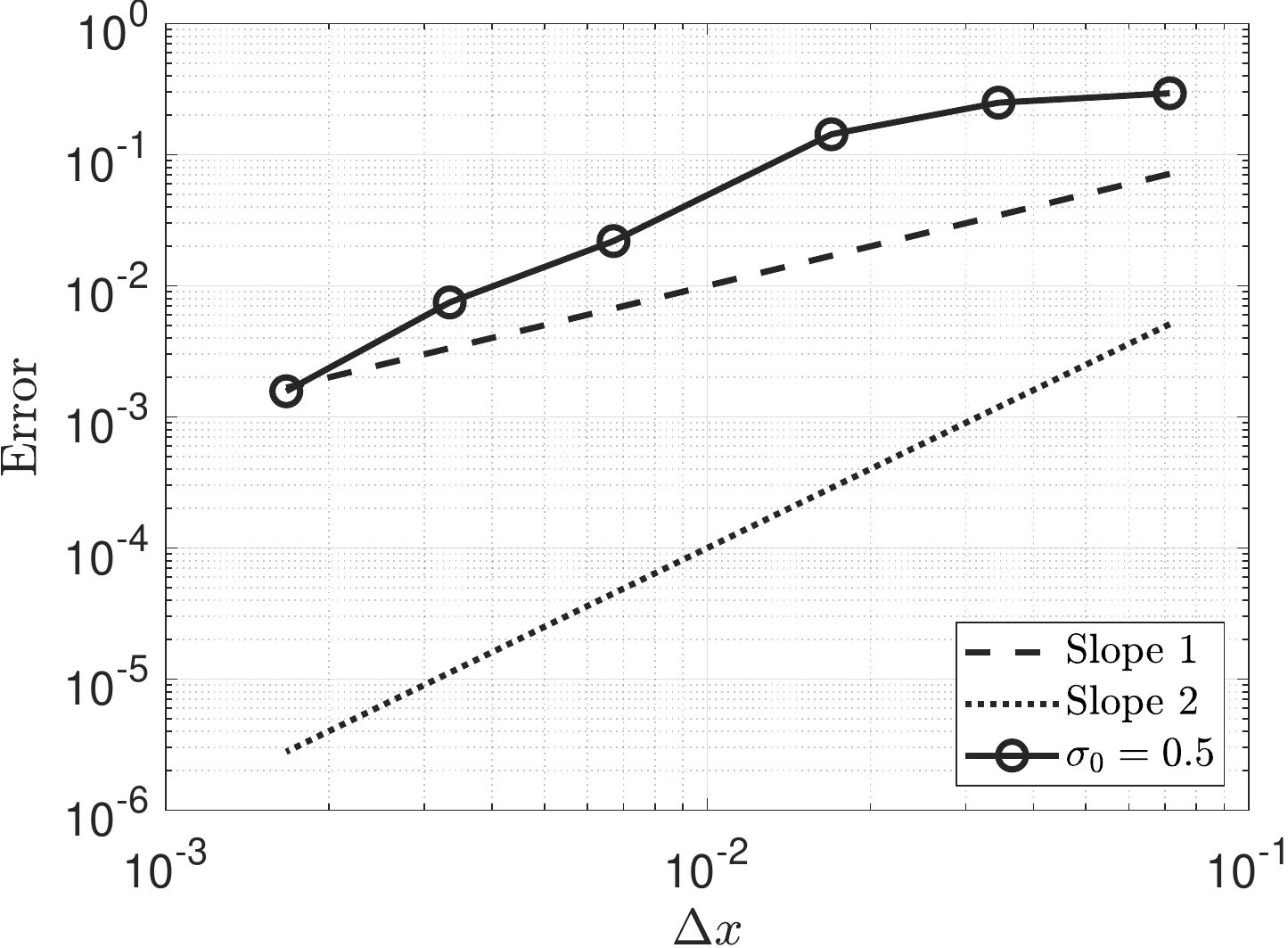}}
		\qquad
		\subfloat[\label{fig:hp}]{\includegraphics[width = 0.29\textwidth]{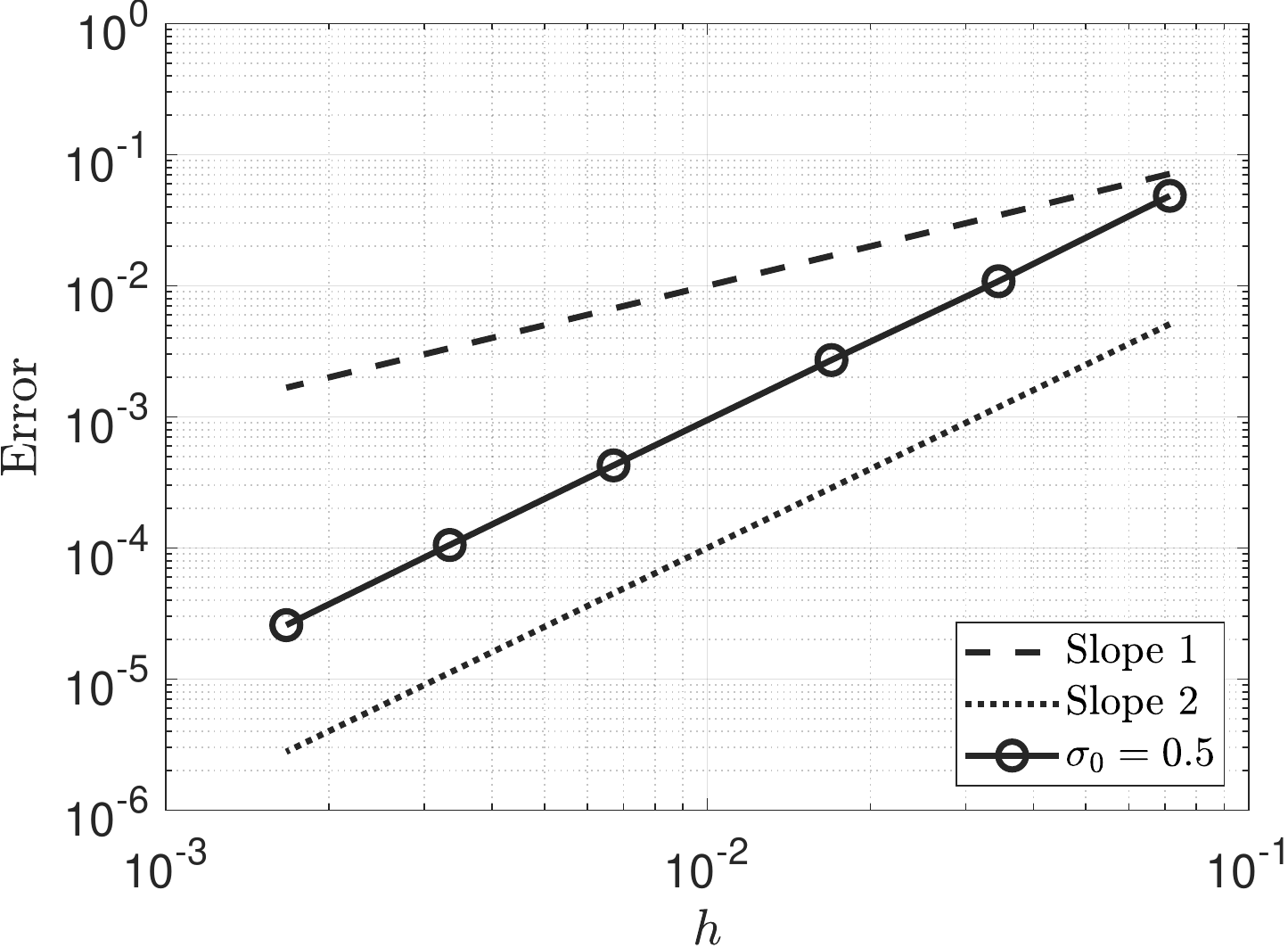}}
		\qquad
		\subfloat[\label{fig:ht}]{\includegraphics[width = 0.29\textwidth]{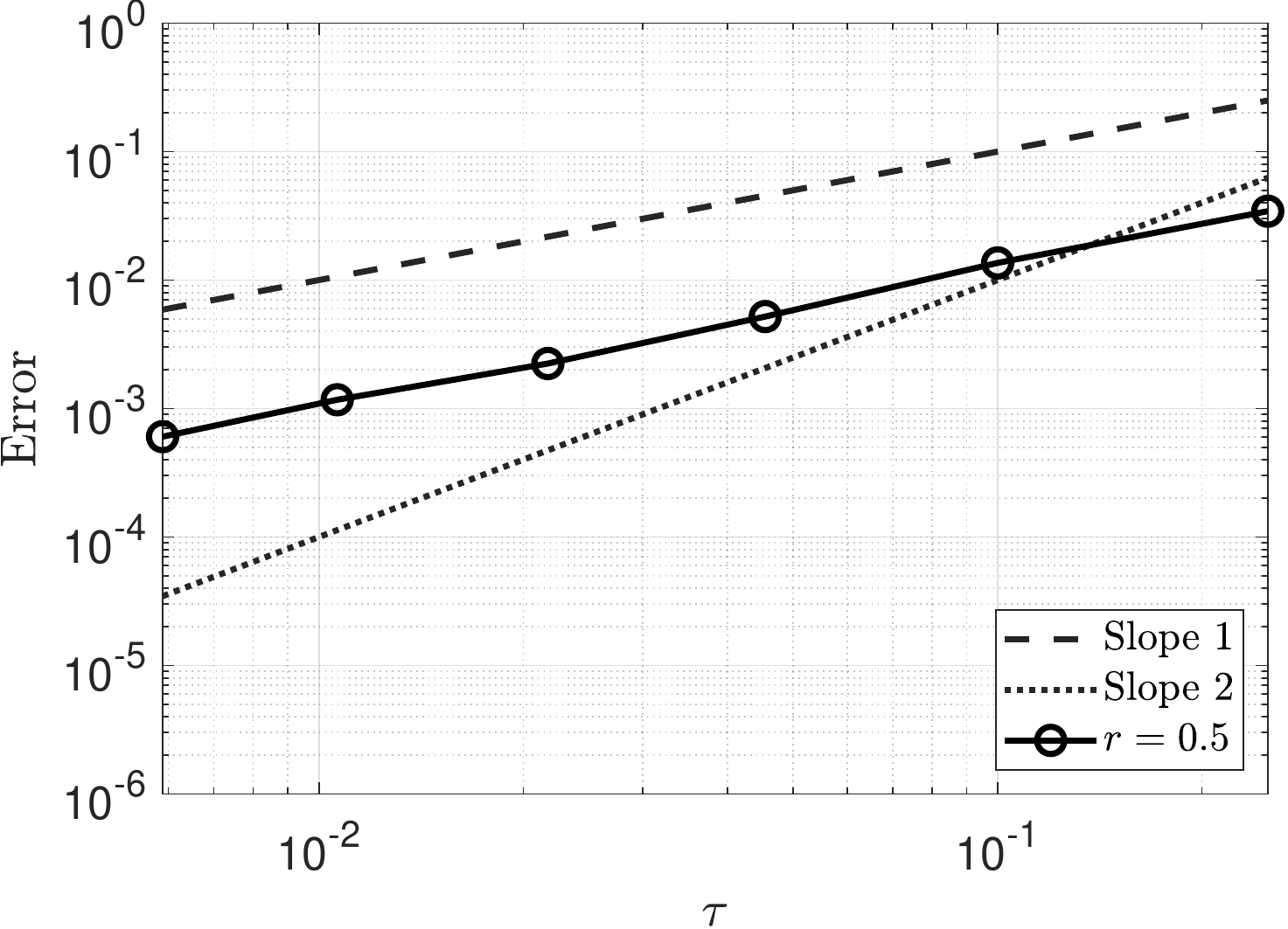}}
		\caption{A log-log plot of the the error wrt. $\hx$, $\hp$, $\hht$ for $\sigma_0 = 0.5$.}
		\label{fig:1}
	\end{figure}
	
	
	\section*{Acknowledgment}
	The authors would like to thank Wolfgang Dahmen for his advice regarding discrete convex envelopes.
	
	\bibliography{MyBibData}
	\bibliographystyle{abbrv} 
	
\end{document}